\documentclass[10pt,twocolumn]{article}
\usepackage[a4paper, margin=1in]{geometry}
\usepackage{sectsty}
\sectionfont{\normalsize}
\subsectionfont{\normalsize}
\subsubsectionfont{\normalsize}

\skip\footins 19.5pt plus 12pt minus 1pt
\usepackage[T1]{fontenc}
\usepackage{caption}
\usepackage{cite} 

\usepackage{hyperref}
\usepackage{xcolor}
\usepackage{xspace}
\usepackage{amsmath,amssymb}
\usepackage{amsthm}
\usepackage{mathtools}
\usepackage{esint}
\usepackage{subfiles}
\usepackage{array}
\usepackage{booktabs}
\usepackage{cleveref} 
\usepackage{multirow}
\usepackage{enumitem}
\usepackage{cuted}
\usepackage{caption}

\makeatletter
\DeclareRobustCommand\onedot{\futurelet\@let@token\@onedot}
\def\@onedot{\ifx\@let@token.\else.\null\fi\xspace}

\def\eg{\emph{e.g}\onedot} 
\def\ie{\emph{i.e}\onedot} 
\def\cf{\emph{cf}\onedot}

\def\etal{\emph{et al}\onedot}

\makeatother

\DeclareMathAlphabet\mathbfcal{OMS}{cmsy}{b}{n}

\definecolor{JungleGreen}{RGB}{40, 169, 143}
\definecolor{BatteryChargedBlue}{RGB}{42, 187, 218}
\definecolor{GargoyleGas}{RGB}{250, 213, 66}
\definecolor{LightCarminePink}{RGB}{236, 94, 100}
\definecolor{MaximumPurple}{RGB}{107, 57, 121}

\newcommand{\R}{\mathbb{R}}
\newcommand{\N}{\mathbb{N}}

\newcommand{\sgn}{\mathrm{sgn\,}}

\DeclareMathOperator*{\argmin}{arg\,min}
\renewcommand{\d}{\,\mathrm{d}}

\providecommand{\dist}{{\mathrm{dist}\,}}
\renewcommand{\div}{{\mathrm{div}\,}}

\newcommand{\surface}{\mathcal{S}}
\newcommand{\normal}{n}
\newcommand{\sdf}{\phi}
\newcommand{\B}{\mathcal{B}}
\newcommand{\domain}{\Omega}

\newcommand{\E}{\mathcal{E}}
\newcommand{\mmheat}{{\E^\phi_{\mathrm{MM}}}}
\newcommand{\mmloss}{{\E_{\mathrm{MM}}}}
\newcommand{\sdfloss}{{\E_{\mathrm{SDF}}}}
\newcommand{\fitloss}{{\E_{\mathrm{fit}}}}
\newcommand{\normalloss}{{\E_{\normal}}}
\newcommand{\orientloss}{{\E_{\B}}}
\newcommand{\mmlossmod}{{\E^{\mathrm{mod}}_{\mathrm{MM}}}}

\usepackage{authblk}
\usepackage[symbol]{footmisc}
\usepackage{wrapfig}

\newtheorem{theorem}{Theorem}[section]
\newtheorem{lemma}[theorem]{Lemma}
\newtheorem{proposition}[theorem]{Proposition}
\crefname{equation}{}{}
\newcommand\CoAuthorMark{\footnotemark[\arabic{footnote}]}

\begin{document}
\title{SDFs from Unoriented Point Clouds using Neural Variational Heat Distances}

\author[1]{Samuel Weidemaier\footnote{These authors contributed equally to this work. (weidemaier@uni-bonn.de, florine.hartwig@uni-bonn.de)}}	
\author[1]{Florine Hartwig\protect\CoAuthorMark}
\author[2]{Josua Sassen}
\author[3]{Sergio Conti}
\author[4]{Mirela Ben-Chen}
\author[1]{Martin Rumpf}

\affil[1]{Institute for Numerical Simulation, University of Bonn}
\affil[2]{Centre Borelli, ENS Paris-Saclay}
\affil[3]{Institute for Applied Mathematics, University of Bonn}
\affil[4]{Technion - Israel Institute of Technology}

\date{}

\maketitle
\begin{strip}
  \includegraphics[width=\linewidth]{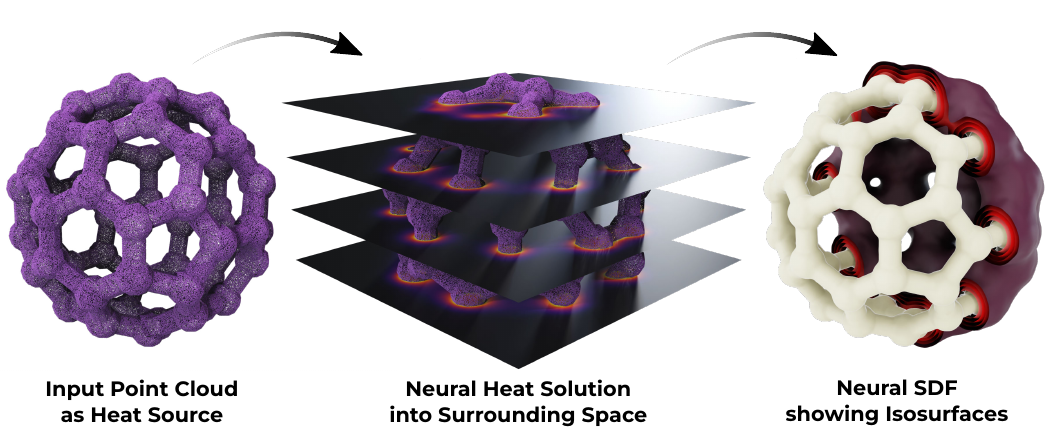}
  \centering
   \captionof{figure}{We compute neural SDFs from unoriented point clouds (left) by first computing a small time step of heat flow (middle) and then using its gradient directions to solve for a neural SDF (right).  }
 \label{fig:teaser}
\end{strip}

\begin{abstract}
	We propose a novel variational approach for computing neural Signed Distance Fields (SDF) from unoriented point clouds. 
	To this end, we replace the commonly used eikonal equation with the heat method, carrying over to the neural domain what has long been standard practice for computing distances on discrete surfaces. 
	This yields two convex optimization problems for whose solution we employ neural networks: 
	We first compute a neural approximation of the gradients of the unsigned distance field through a small time step of heat flow with weighted point cloud densities as initial data. 
	Then we use it to compute a neural approximation of the SDF.  
	We prove that the underlying variational problems are well-posed. 
	Through numerical experiments, we demonstrate that our method provides state-of-the-art surface reconstruction and consistent SDF gradients. 
	Furthermore, we show in a proof-of-concept that it is accurate enough for solving a PDE on the zero-level set.
\end{abstract}

\section{Introduction}
Neural implicit representations have increased in popularity in recent years~\cite{essakine2025where}, due to their flexibility and favorable performance in applications such as shape reconstruction~\cite{yunus2024recent}. While many volumetric functions can represent a given surface, neural \emph{Signed Distance Fields} (SDFs)~\cite{park2019deepsdf} have been favored among this family of functions due to their applicability beyond reconstruction, \eg for constructive solid geometry~\cite{marschner2023constructive} and collision detection~\cite{liu2024real}. Furthermore, SDFs enable the use of \emph{level set methods} for solving PDEs, which have a myriad of applications in the classical setting~\cite{osher2001level,droske2004level}, and have been explored only recently in the neural setting~\cite{mehta2022level}. 

A neural SDF must fulfill a few desired properties to be usable in PDE level set methods. First, the SDF should accurately represent the surface, namely its zero-level set should be well resolved to encode surface detail. In addition, the non-zero level sets should be accurate in a \emph{narrow band} near the surface to enable the accurate computation of differential quantities, such as the surface gradient and the (weak) Laplace-Beltrami operator. Finally, for practical usability, it is advantageous that the SDF can be constructed directly from an \emph{unoriented} point cloud. Existing approaches usually either require an approximated ground truth distance~\cite{novello2022exploring} or learned priors~\cite{chou2022gensdf}. Other approaches either represent the zero-level set well but do not generate an accurate SDF in a narrow band~\cite{wang2023neural}, or generate an SDF with a low resolution surface representation~\cite{coiffier20241}.

Computing distances is a difficult problem, not only in the neural SDF setting. The governing equation is the eikonal equation, namely the gradient of the SDF has unit norm, which is well known to be challenging to solve in many scenarios~\cite{sethian1999fast}. Essentially, the eikonal equation has many solutions, of which only the \emph{viscosity solution}~\cite{Crandall1997} is the SDF. Since the neural net is optimized using gradient descent, it can converge to a non-SDF local minimum, depending on the initialization. These inherent challenges with the eikonal equation may cause difficulties for neural methods that use it in their loss functional~\cite{sitzmann2020implicit,ben2022digs}.
Instead of directly using the eikonal loss, the heat method~\cite{crane2013geodesics} has been proposed for computing geodesic distances on surfaces and has been successfully used in multiple applications~\cite{feng2024heat}. Its success follows directly from the simplicity of the method, requiring the solution of two well-posed \emph{elliptic} PDEs, which are considerably better behaved than the eikonal equation. 

We propose to lift the heat method to the neural setting (see~\Cref{fig:teaser}). To that end, we define two well-posed variational problems, which on the one hand, can be optimized well with stochastic gradient descent using a standard network architecture, and on the other hand, have existence guarantees. As in the heat method, in the first step we solve for a time step of the heat flow and compute using this the gradient of the \emph{unsigned} distance function, exploiting the relation between the heat kernel and the distance. 
Unlike in the heat method, we do not have a discretization of the domain, as we work with neural representations. Thus, to compute the gradients, we use a \emph{variational} formulation for a backward Euler timestep of the heat flow~\cite{thomee2007galerkin} in the spirit of the general class of minimizing movement schemes\cite{park2023deep}, see~\Cref{fig:heat_heat_slices} (center).
\begin{figure}[b]
	\centering
	\includegraphics[height=0.295\linewidth]{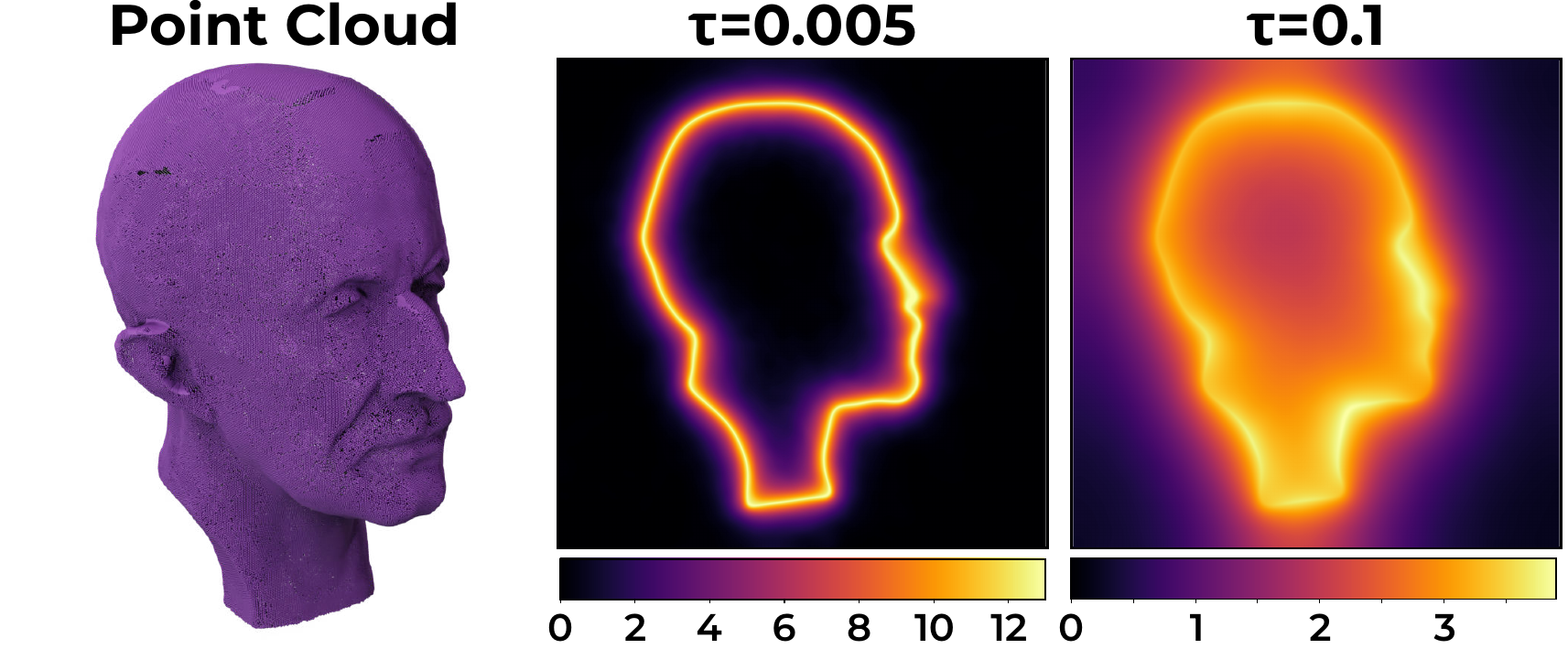}
	\includegraphics[height=0.295\linewidth]{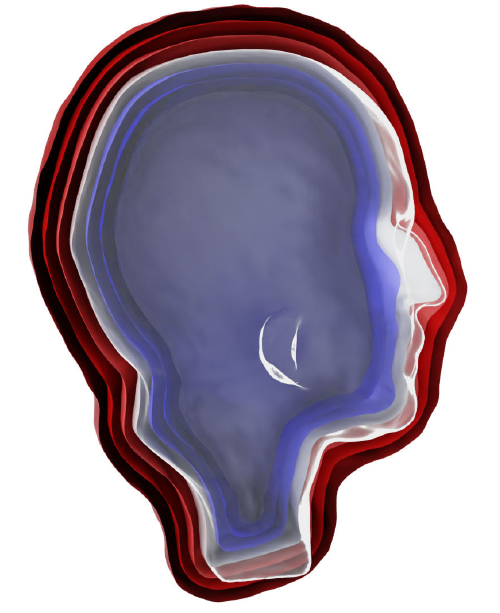}
	\caption{Point cloud input data (left), a sliced color coding of two different time steps of the heat flow (middle),
		equispaced isosurfaces for positive (red) and negative (blue) values of the SDF (right). }
	\label{fig:heat_heat_slices}
\end{figure}
In the second step, we fit the computed gradients to obtain the signed distance, taking into consideration that the input point cloud is \emph{unoriented}, leading to accurate level sets in a narrow band near the surface; see~\Cref{fig:heat_heat_slices} (right).

We show that compared to state-of-the-art methods, our approach, denoted HeatSDF, leads to a good balance between the fidelity of the zero-level set to the surface and the accuracy of the SDF in a narrow band near the surface. Furthermore, we show that our method works well on point clouds
with non-uniform density. In addition, we prove existence of solutions of the variational problems in both steps of our method. Finally, we provide a proof-of-concept for using this approach to solve a PDE directly on the neural field and do geometric queries demonstrating that the generated SDF is sufficiently accurate for these challenging tasks.

\subsection{Related Work}
\paragraph*{Neural SDFs.}
Geometric deep learning is a widely researched field~\cite{bronstein2021geometric}, and within it we focus on representing a \emph{single} surface embedded in $\R^3$, as opposed to learning how to perform a specific task, \eg, segmentation or classification, on a \emph{dataset} of models.
Surfaces are represented using a \emph{neural implicit function}, namely a network that encodes a function $\phi_\surface \colon \R^3 \to \R$, whose zero-level set is the required surface $\surface$. A surface can be represented using many neural implicits~\cite{essakine2025where}, however, the one often sought after is the \emph{Signed Distance Field} (SDF), where the function encoded is the signed distance to the surface. Neural implicits, which are SDFs are useful in geometric applications beyond surface reconstruction, \eg, solving PDEs and geometric flows~\cite{mehta2022level} and performing constructive solid geometry on neural surfaces~\cite{marschner2023constructive}. Especially for the solution of PDEs using level set methods, it is important that the SDF is accurate not only \emph{on} the surface but also in a \emph{narrow band} near the surface. 

A full survey of neural SDFs is beyond our scope, and we refer the reader to a recent review on the topic~\cite{schirmer2024geometric}. Here, we will focus on methods that are unsupervised (i.e., do not use an approximated ground truth SDF), do not learn priors from a dataset, and do not require normal orientation. This is one of the most flexible starting points, where an SDF can be computed given only an unoriented point cloud.

The seminal paper SIREN~\cite{sitzmann2020implicit} uses an eikonal loss, which combined with periodic activation functions, leads to mostly good results. As the results are very initialization dependent, DiGS~\cite{ben2022digs} has proposed a geometric initialization that improves the result for unoriented point clouds, as well as an additional regularization with a divergence-based loss. StEik~\cite{yang2023steik} has discussed the instabilities inherent in the eikonal equation and proposed additional regularizing losses, as well as using a \emph{quadratic} neural network instead of the usual linear one. Concurrently with our work, HotSpot~\cite{wang2025hotspotsigneddistancefunction}, building on the quadratic network of StEik, has proposed adding a regularizing loss that encourages the exponent of the SDF to fulfill the heat equation. This approach, like ours, incorporates the relation between the distance function and the heat equation. However, their formulation is still based on the eikonal equation, making it non-convex and more difficult to optimize, whereas we propose to solve two \emph{convex} functionals successively. Furthermore, the density of the point cloud is not taken into account in their method, leading to problematic regions for non-uniform point cloud sampling.

Additional approaches include the singular-Hessian regularizer~\cite{wang2023neural}, which enforces the Hessian of the SDF to be singular near the surface, and the $1$-Lipschitz approach~\cite{coiffier20241}, which uses a hinge-Kantorovitch-Rubinstein loss and requires oriented normals. The first achieves excellent surface reconstruction error but usually has high eikonal error even in a narrow band around the surface, and the second leads to accurate SDFs in a narrow band, but the resulting surface is overly-regularized, losing geometric detail. Our method, on the other hand, leads to a good balance between these two requirements, such that the reconstructed surface is detailed, and the SDF is accurate enough in a narrow band near the surface to solve a neural surface PDE. Furthermore, in contrast to the $1$-Lipschitz approach, which uses a tailored network architecture, 
our method works with a standard network architecture. 

\paragraph*{Variational Neural PDEs.} Minimizing movements schemes are a robust and effective time discretization for 
gradient flow type evolution problems in metric spaces \cite{de1993new}. In \cite{park2023deep},
deep learning-based minimizing movement schemes for a variety of PDEs of gradient flow type are proposed. 
In \cite{hu2024energetic}, a structure-preserving Eulerian neural network discretization for the numerical solving of 
$L^2$-gradient flows is presented. These schemes are based on the variational principle of the energy-dissipation law, which ensures 
the monotonic decay of the free energy and avoids unphysical states.
We take inspiration from these machine learning approaches to properly capture the gradient flow dynamics for the heat equation in Euclidean and Riemannian spaces.

\paragraph*{Heat-based distances.}
\setlength{\columnsep}{6pt}%
Using the heat equation for computing approximate geodesic distances has been introduced to geometry processing by Crane et al.~\cite{crane2013geodesics,crane2017heat}. 
This method is based on two main steps. 
First, diffusing heat for a short time step starting from heat concentrated at the source points. 
Then, the normalized gradients of the heat function are used to compute the approximate geodesic distances by solving a Poisson equation. 
Many generalizations and improvements have been proposed, the latest being generalization to \emph{signed} distance functions (GSD)~\cite{feng2024heat} by diffusing \emph{signed} normals. 
To the best of our knowledge, heat-based distance approximations have not been used in the neural setting so far. 
Following the original heat method, we also diffuse a scalar function from the input sources. 
However, we incorporate additional constraints in the second step, which allows us to correctly orient the heat gradients. 
Hence, unlike GSD, we are not diffusing normals but start with unoriented point clouds to compute a \emph{signed} distance field.
A more detailed comparison to GSD can be found in \Cref{sec:grid_comparison}.

\subsection*{Contributions}
Our main contribution is a novel approach for computing a neural SDF from an unoriented point cloud using the heat method. Specifically, we
\begin{itemize}
	\item provide a theoretically sound and robust variational scheme to approximate first the field of level set normals
	and based on that, in the second step, the signed distance,
	\item compute neural SDFs which are accurate not only \emph{on} the surface, but also in a narrow band \emph{near} the surface,
	\item handle spatially varying point cloud densities,
	\item exemplify the application to neural PDE solving on level sets.
\end{itemize}

\section{Method}
Throughout this paper, we consider a surface $\surface$ which is the boundary of a Lipschitz set, contained in $[-1,1]^3$
such that the complement of this set is connected. Later, we will assume additional smoothness of $\surface$.
Following the original heat method~\cite{crane2013geodesics}, our method comprises two steps. 
First (\Cref{sec:part1}), we compute the solution 
$u\colon\R^3\to\R$ of a small time step of the heat equation with the mean surface measure--reflecting the density of the point cloud--as initial data. This measure is approximated via a 
weighted sum of point measures of a not necessarily uniformly distributed input point cloud. 
In the second step (\Cref{sec:part2}), we compute the SDF $\sdf\colon\R^3\to\R$, by requiring that (1) $\vert\sdf\vert$ is small on the set of input points, (2) the direction of $\nabla \sdf$ is aligned with the direction of $\nabla u$. 
To allow unoriented point clouds as input, we automatically find regions $\B^+$ and $\B^-$ 
on the two sides of the cloud,
and encourage $\sdf$ to have the correct sign in these regions. Both $u$ and $\sdf$ are parameterized with a neural network (see \Cref{sec:implementation} for the details). The regions $\B^\pm$ are computed using a simple deterministic algorithm.

\subsection{Unoriented Point Cloud to Unoriented Normals}
\label{sec:part1}In the first step, we use that the normalized gradients of a short-time solution of the heat equation on $\domain$, with the mean surface area density of a smooth surface as initial data,
approximate the gradient of the distance function $\dist(\cdot, \surface)$ in the vicinity of the surface (cf.~Crane et al \cite{crane2013geodesics} for the case on manifolds).
We restrict here to a computational domain $\domain \coloneqq (-1.2,1.2)^3$ where the surface $\surface$ is scaled to fit into $[-1,1]^3$.
In the simple case, for initial data $u^0\in L^2(\domain)$ a single timestep $\tau$ of the
heat equation 
\begin{align} \label{eq:heat_eq} \partial_t u -\Delta u =0
\end{align}
with initial data $u(0) =u^0$ 
can be approximated with a backward Euler scheme 
\begin{align}\label{eq:BEuler}
	\tfrac{u^\tau-u^0}{\tau} - \Delta u^\tau = 0
\end{align}
with $u^\tau\approx u(\tau,\cdot)$.
We want to parametrize the solution \(u\) using a neural network. Here, we prefer a variational approach to compute the solution to 
\cref{eq:BEuler} over a least-squares loss for the residual of \cref{eq:heat_eq}, as this would decrease the numerical stability. To this end, we introduce the energy 
\begin{align}\label{eq:MMsimple}
	\mmloss(u) \coloneqq
	\int_{\domain} (u-u^0)^2 + \tau \vert \nabla u  \vert^2  \d x\,.
\end{align}
Indeed, the associated Euler-Lagrange equation characterizing the minimizer $u^\tau  \coloneqq   \argmin_{u} \mmloss(u)$ 
is given by
\begin{align*}
	0 = \partial_u \mmloss[u^\tau](\vartheta) =
	2 \int_{\domain} (u^\tau -u^0)\vartheta  + \tau \nabla u^\tau \cdot \nabla \vartheta \d x
\end{align*}
for all test functions $\vartheta \in H^{1}(\domain)$, which is the weak formulation of \eqref{eq:BEuler}.
Here, $H^{1}(\domain)$ is the space of $L^2$ functions on $\domain$ with weak derivatives also in $L^2$.
We remark that this corresponds to a simple example of a minimizing movements approach for the solution of \cref{eq:heat_eq}
\cite{de1993new} (cf.~\cite{park2023deep} in the context of neural networks).

The case of interest in this paper is the mean two-dimensional surface measure as initial data $u^0$, which is a measure and not an element of $L^2(\Omega)$. 
The application of this measure to a function $\phi$
is defined as $\fint_\surface \phi \d a \coloneqq \left(\int_\surface \d a\right)^{-1} \int_\surface \phi \d a$.
When $\surface$ is the boundary of a Lipschitz set, this measure is well-defined on $L^2$ functions on $\surface$ and thus
by the trace theorem on $H^{1}(\domain)$.

In order to apply a similar variational approach for the solution of \cref{eq:BEuler} in case of a measure as initial data we need to modify the energy $\mmloss$.
We expand $\int_{\domain} (u-u^0)^2 \d x$ and skip the term $\int_\Omega (u^0)^2 \d x$, which is irrelevant for the minimization 
since it does not depend on the variable $u$. Thus, we consider the modified minimizing movement energy
\begin{align}\label{eq:minimizing_movements_surface}
	\mmlossmod(u) \coloneqq
	\int_{\domain} u^2 \d x - 2 \fint_\surface  u \d a + \tau \int_{\domain} \vert \nabla u  \vert^2  \d x .
\end{align}
Defining as above $u^\tau  \coloneqq \argmin_{u} \mmlossmod(u)$,
the associated Euler-Lagrange equation for $u^\tau$ is given by
\begin{align}\label{eq:MMEL}
	0 = \int_{\domain} u^\tau \vartheta \d x - \fint_\surface \vartheta \d a +
	\tau \int_{\domain} \nabla u^\tau \cdot \nabla \vartheta \d x
\end{align}
for all $\vartheta \in H^{1}(\domain)$, which is the adapted weak form of \eqref{eq:BEuler}.
The following proposition provides the existence of the discrete heat solution $u^\tau$.
\begin{proposition}[Existence and uniqueness of a heat time step]\label{prop:ex}
	Under the above assumptions
	there exists a unique minimizer $u^\tau$ on $H^{1}(\domain)$, which solves $0 = \partial_u \mmlossmod(u)$.
\end{proposition}
The proof can be found in \Cref{appproofs}.

As explained before, 
the normalized gradient
\begin{align}\label{eq:ntau}
	\normal^\tau \coloneqq - \frac{\nabla u^\tau}{\lvert \nabla u^\tau \rvert},
\end{align}
which we assume to be well-defined almost everywhere,
yields for smooth surfaces $\surface$ a good approximation for the gradient of the distance function, 
\ie the unoriented normal field of \(\surface\), which is well-defined up to the medial axis.
However, note that this gradient is consistent only in the near field of the surface
and corresponds to the \emph{unsigned} distance and not the \emph{signed} we actually would like to compute (\Cref{fig: visuheat}).

\begin{figure}[b]
	\centering
	\includegraphics[width=\linewidth]{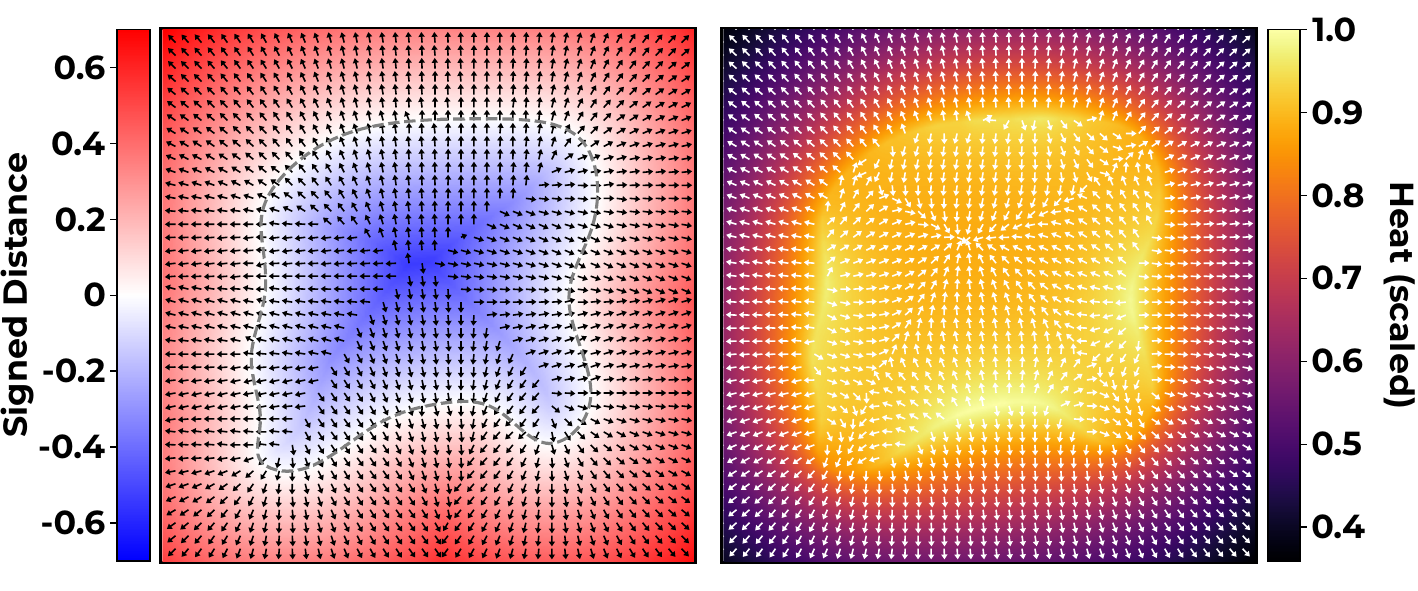}
	\caption{Comparison of gradient directions between 
		a signed distance of a (hyper)surface in 2D (left), and the solution of a small time step of the heat equation with
		mean surface measure as initial data (right).\label{fig: visuheat}
	}
\end{figure}

\subsection{Unoriented Normals to SDF}
\label{sec:part2}
In the second step, we assume that we are given an approximate, unoriented normal field \(\normal^\tau\) and aim for the SDF $\sdf$ of the surface. 
To this end, we pick up the ansatz by Crane \etal \cite{crane2013geodesics} and consider a modification of the second step of their heat method. 
We construct an SDF approximation from $\normal^\tau$  with \(\sdf <0\) inside the surface, where the gradient should point 
in the opposite direction of \(\normal^\tau\), and conversely, \(\sdf > 0\) outside the surface with the gradient pointing in the same direction as \(\normal^\tau\) (\Cref{fig:sketchSigns}). 
To this end, we seek a minimizer of the normal fitting objective
\begin{equation}
	\normalloss(\sdf) \coloneqq \int_{[\sdf < 0]}  \lvert \nabla \sdf + n^\tau \rvert^2\d x + \int_{[\sdf \ge 0]}  \lvert \nabla \sdf - n^\tau \rvert^2\d x,
\end{equation}
where \([\sdf < 0] = \{ x \in \domain \mid \sdf(x)<  0\}\) and  \([\sdf \geq 0] = \{ x \in \domain \mid \sdf(x)\geq  0\}\).
Here, to have a decomposition of the domain, we arbitrarily 
included the set $\{\phi=0\}$  in the second term, which does not make a difference since this set is expected to be lower dimensional.
Furthermore, we want the zero-level set of \(\sdf\) to coincide with the given surface, \ie we take into account the surface fitting objective
\begin{equation} 
	\fitloss(\sdf) \coloneqq \int_\surface \sdf^2 \d a.
\end{equation}

The sum \(\normalloss + \fitloss\) has at least four minima: the signed distance, the unsigned distance, and their respective negatives. 
Hence, we need an additional term to favor the SDF as the desired solution.
For this, we assume to be given sets \(\B^-\) on the inside of \(\surface\) and \(\B^+\) on the outside, cf. \Cref{sec:Bplusminus} on 
how to generate these sets automatically.
Now, we define the orientation objective 
\begin{equation}
	\orientloss(\sdf) \coloneqq \int_{\B^-} \chi_{[\sdf > 0]} \d x + \int_{\B^+} \chi_{[\sdf < 0]} \d x,
\end{equation}
which promotes that \(\sdf \leq 0\) in \(\B^-\) and \(\sdf \geq 0\) in \(\B^+\). Here, we denote by $\chi_{[\sdf > 0]}$ and $\chi_{[\sdf < 0]} $ the characteristic functions of the respective sets. That is, $ \chi_{[\sdf > 0]}(x)=1$ if $\sdf(x)>0$ and $0$ otherwise; similarly for $\chi_{[\sdf < 0]}$.
Finally, our overall objective for computing the SDF becomes
\begin{equation}
	\label{eq:overall_objective}
	\sdfloss(\sdf) \coloneqq \normalloss(\sdf) + \lambda_{\mathrm{fit}}  \fitloss(\sdf) + \lambda_{\B}  \orientloss(\sdf).
\end{equation}
In fact, it will turn out that, in our descent scheme, $\orientloss$ vanishes after a few descent iterations, which 
allows us to choose $\lambda_{\B}=1$.
We obtain the following existence result, whose proof can be found in
\Cref{appproofs}.
\begin{theorem}[Existence of minimizers for $\sdfloss$] \label{thm:sdl}\ \\
	Suppose the normal field
	$n^\tau$ is measurable, and two Borel sets $\B^+$ and $\B^-$ in $\domain$ are given,
	then there exists a minimizer of the signed distance loss $\sdfloss$ in $H^{1}(\domain)$.
\end{theorem}

\begin{figure}
	\centering
	\includegraphics[width=\columnwidth]{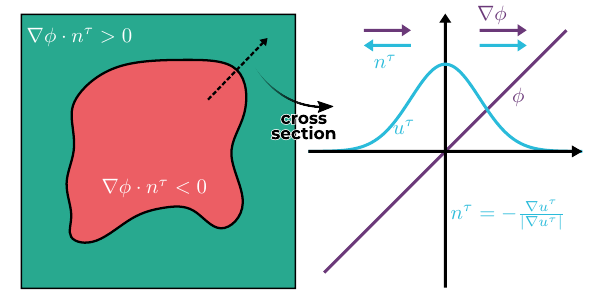}
	\caption{Sketch of the heat solution \(u^\tau\) and the SDF \(\phi\) on a one dimensional cross section. Inside the (hyper)surface, the SDF gradient \( \nabla \phi\) and the normalized gradient \( n^\tau\) of the heat time step point in opposite directions; outside, they point in the same direction. \label{fig:sketchSigns}}
\end{figure}
\section{Implementation} \label{sec:implementation}
We require some additional ingredients for both steps. For the heat method (\cf \Cref{sec:heat_details}), we describe the quadrature rules for numerically approximating the integrals incorporating the point cloud density. Additionally, we introduce a  "far field" heat solution that allows us to use a very small time step while still obtaining non-zero heat throughout the entire domain. For the SDF computation (\cf \Cref{sec:sdf_details}), we explain the definition of the regions $\B^\pm$ that generalize the heat method to unoriented point clouds and their incorporation in the loss function using a smoothed characteristic function.
\begin{figure}[b]
	\centering
	\includegraphics[width=\linewidth]{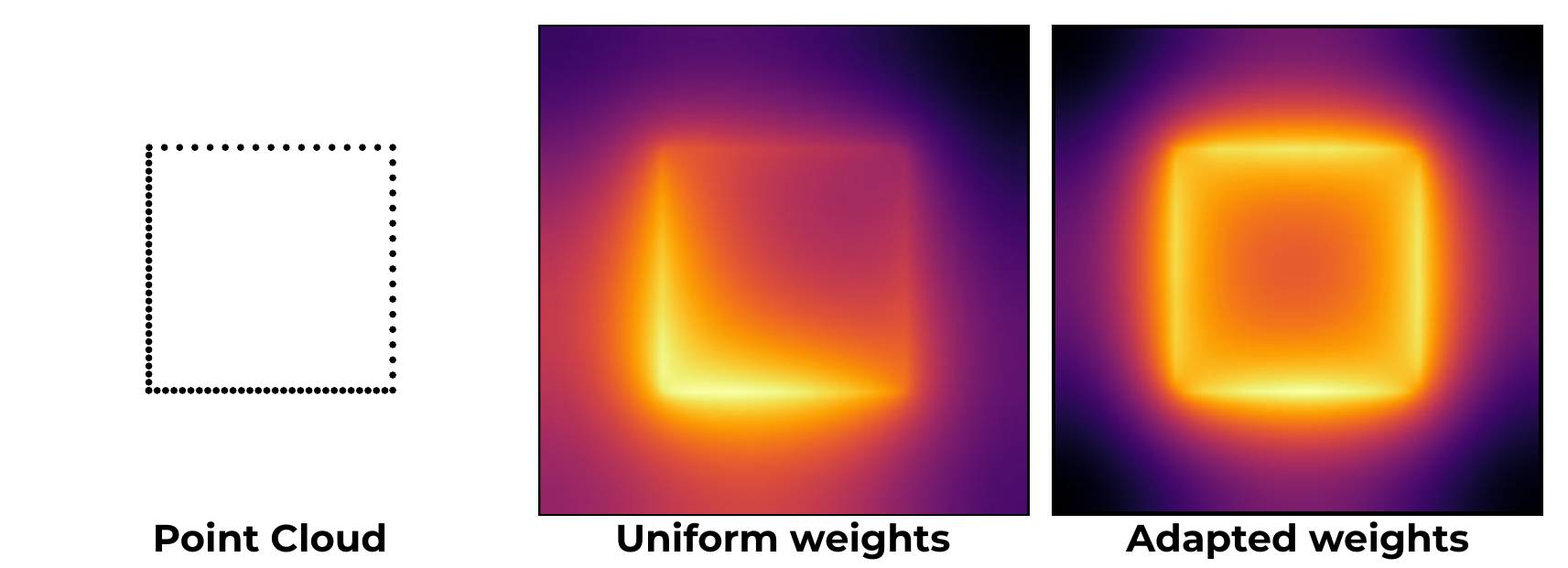}
	\caption{Left: non-uniform input point cloud on a square (qualitative visualization), middle: heat solution $u^\tau$ in two dimensions 
		using uniform weights for the mean (hyper)surface integral \Cref{eq:weights} and right: heat solution using the locally  adapted 
		weights \Cref{eq:uniformweights}. \label{fig:2Dweights}}
\end{figure}

\subsection{Heat Solution}
\label{sec:heat_details}
\subsubsection{Quadrature \& Point Clouds}
To evaluate the loss functions, we need of course, to numerically approximate the involved integrals.
To this end, we scale all the examples to  \( [-1,1]^3\) (with a distance of at least $0.2$ to the boundary of $\domain$).
Then, in every iteration of the descent algorithm, we uniformly sample points from \(\Omega\) and approximate the integral by the average of the integrand's values at these sampled points. This number of quadrature points was set to $10000$ for all experiments.
The first step of the algorithm involves the evaluation of the functional $u\mapsto \fint_\surface u \d a$.
The surface $\surface$ is described by a point cloud $(x_i)_{i=1,\ldots, N} \subset \R^3$.
In case the point cloud is uniformly distributed on $\surface$, one might approximate the mean surface integral by
\begin{align}
	\fint_\surface u \d a \approx \frac{1}{N} \sum_{i=1}^N u(x_i). \label{eq:uniformweights}
\end{align} 
In the general case of a non-uniform distribution of the points on the surface $\surface$, we use the following local averaging strategy.
Let $\nu\colon\R^3\to \R$ be a compactly supported, non-negative, smooth function, \eg, $\nu(s) = \exp((\vert s\vert^2-1)^{-1})$ for $\vert s\vert < 1$ and $0$ else.
Then, for $\nu^\epsilon (s) = \epsilon^{-3} \nu(\epsilon^{-1} s)$ we first define the non-normalized local adaptive weights 
$$
\tilde \omega^\epsilon_i \coloneqq \left(\sum_{j=1,\ldots, N} \nu^\epsilon (x_i-x_j)\right)^{-1}$$
and then normalize these by choosing $\omega^\epsilon_i  =\frac{\tilde \omega^\epsilon_i}{\sum_{j=1}^{N} \tilde \omega^\epsilon_j}$.
Finally, we obtain 
\begin{align}
	\fint_\surface u \d a \approx  \sum_{i=1}^N \omega^\epsilon_i u(x_i) \label{eq:weights}
\end{align}
as a density rescaled approximation of the integral mean of $u$ on $\surface$. In practice, $\epsilon$ is chosen such that every $\epsilon$-ball centered at any point contains at least 12 neighboring points. In \Cref{fig:2Dweights}, we show the heat solution $u^\tau$ for a non uniform input point cloud comparing uniform weights \Cref{eq:uniformweights} with the local adaptive weights \Cref{eq:weights}. 
It demonstrates that a proper weighting is necessary for non-uniform point clouds to obtain accurate normal directions from the heat solution.

\begin{figure}[t]
	\includegraphics[width=\linewidth]{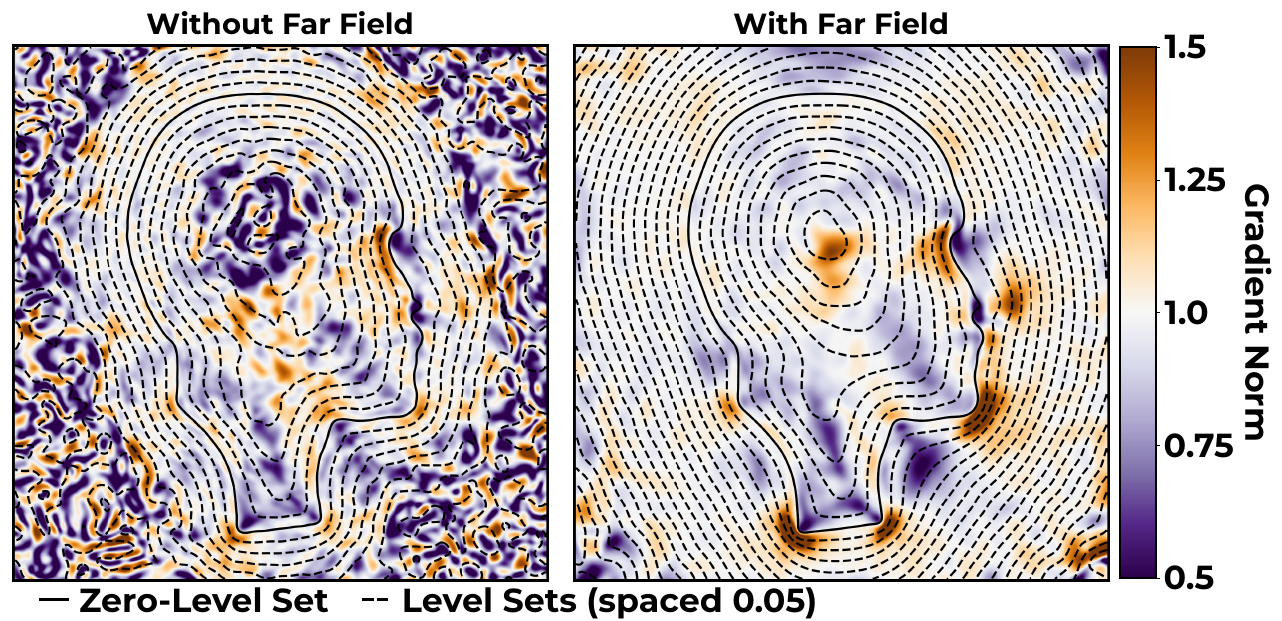}
	\caption{
		Level sets and gradient norms on a slice of the Max Planck head, where the  SDF is computed without far field (left) and with a near--far field 
		blending of the normalized gradient fields of the heat diffusion (right). \label{fig:farfieldstudy}
	}
\end{figure}

\subsubsection{Far Field Normals}
To obtain a good approximation of surfaces with highly curved features, we use small time steps \(\tau\) 
in the variational heat flow time step \Cref{eq:minimizing_movements_surface}  to obtain a consistent field of normals $\normal^\tau$.
However, we observed that for such small time steps, the heat solution \(u^\tau\) has very small 
gradient norms $\lvert\nabla u^{\tau}\rvert$ further away from the surface $\surface$.
This leads to a lack of robustness when normalizing $\nabla u^\tau$ to compute $\normal^\tau$. 
To mitigate this, we compute a second solution $u^{\hat\tau}$ of the variational heat flow problem 
\Cref{eq:minimizing_movements_surface} for a significantly larger time step size $\hat\tau$ and use the resulting less precise normal field 
$\normal^{\hat\tau}$ in the far field. 
To still obtain an overall smooth normal field, we blend the gradient fields of the heat solution around a critical value \(\kappa\) of the solution $u^\tau$.
Hence, we define the blended normal field
\begin{align}\label{eq:blendednormals}
	\normal^\ast(x) = \frac{(1-\beta_\kappa(u^\tau(x))) \nabla u^\tau(x) + \beta_\kappa(u^\tau(x)) \nabla u^{\hat\tau}(x)}
	{\lvert (1-\beta_\kappa(u^\tau(x))) \nabla u^\tau(x) + \beta_\kappa(u^\tau(x)) \nabla u^{\hat\tau}(x)\rvert},
\end{align}
for $\beta_\kappa(s) = \mu(\tfrac{s}{\kappa})$ using the $C^1$ smooth, cubic Hermite blending function
$\mu\colon\R\to \R_{\geq0}$ with $\mu(s)=1$ for $s<0$, $\mu(s)=\tfrac14 (2s+1)(2s-2)^2$ for $0 \leq s\leq 1$ and $\mu(s)=0$ for $s>1$. In fact, for $u^\tau \geq \kappa$ we have $	\normal^\ast(x) = n^\tau(x)$ and in the limit for decreasing values of $u^\tau $ we obtain  $\normal^\ast(x) = n^{\hat\tau}(x)$.
We simply use $\normal^\ast$ instead of $\normal^\tau$ in the smoothed normal approximation objective \eqref{eq:overall_objective}.
When minimizing $\normalloss$ for this normal field, we combine  
an accurate signed distance in the near field and a proper identification of the zero-level set 
with a robust extension of the signed distance in the far field. In the applications we choose $\kappa =\tfrac35 \max_{u^\tau}$,
where $\max_{u^\tau}$ is the maximum of $\vert u^\tau \vert$ on the cell centers of the regular grid used below
to distinguish inside and outside of $\surface$ (\cf \Cref{sec:sdf_details}).

The results in \Cref{fig:farfieldstudy} indeed show that using the blended normal field leads to much better gradient norms and level sets farther from the surface.
However, they degrade near the medial axis of the shapes. Thus creating a trade-off. 
If one is only interested in accurate results in a narrow band, the near field is sufficient. For our quantitative evaluations in the comparison section, we have always used the far field approach.

\subsection{SDF from Heat Solution}
\label{sec:sdf_details}
\subsubsection{Inside \& Outside}\label{sec:Bplusminus}
We construct the sets \(\B^-\) on the inside of \(\surface\) and \(\B^+\) on the outside using the following simple and efficient algorithm (see \Cref{figureinsideoutside}).
We consider a regular grid with edge length $h>0$ and cell centers $\alpha h$ with $\alpha\in \mathbb{Z}^3$ and denote the corresponding open cells by $B_\alpha$.
Step 1. We mark all cells whose closure has a non-empty intersection with $\surface$ (\ie, containing points of the given point cloud) as interfacial cells.
Step 2. Beginning with one cell intersecting the boundary of $\domain$ that is not marked as interfacial,
we successively mark each of the six neighboring cells for every outside cell $B_\alpha$, which are not yet marked, as outside.
Step 3. We mark all remaining cells as interior. 
Finally, we define $\B^+$ as the union of all outside cells and \(\B^-\) as the union of all inside cells.
In practice, a rather coarse grid is sufficient, \eg, we used $64^3$ grid points and grid size $h=0.0375$. We note that for a very sparse point cloud, we may require a larger $h$. For example, for the sparse bean model in \Cref{tab:beanexperiment}, we used $2h$ as the grid size, which was the only case where different parameters were needed. 
\begin{figure}[b]
	\centering
	\includegraphics[width=\linewidth]{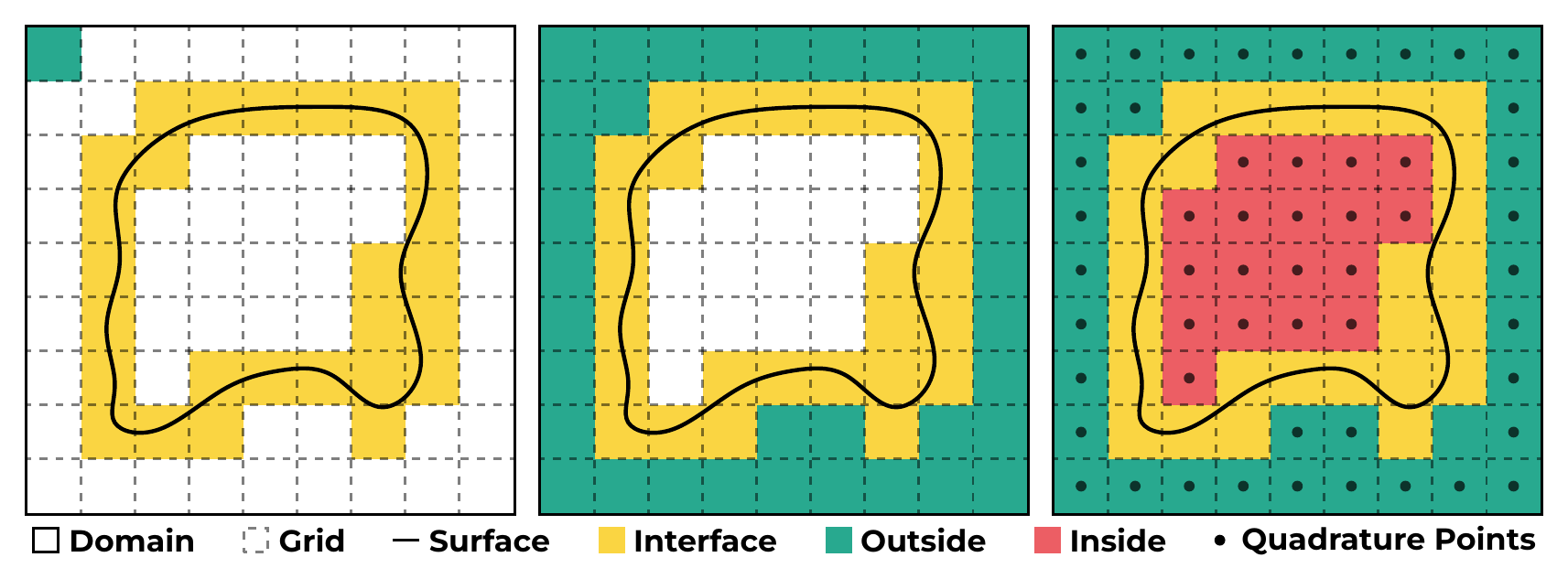}
	\caption{
		Sketch of the algorithm to determine \(\B^\pm\): 
		step 1. initial marking of interfacial cells, step 2. iterative marking of outside cells, step 3. marking all remaining cells as inside.
		\label{figureinsideoutside}
	}
\end{figure}

\subsubsection{Smoothed Characteristic Functions}
The normal approximation objective \(\normalloss\) depends on the sets \([\sdf < 0]\) and \([\sdf > 0]\), which means that it depends on the function \(\sdf\) in a non-smooth manner due to the thresholding operation.
To facilitate the numerical minimization of \(\normalloss\), we introduce a smooth approximation of the corresponding characteristic functions such that the objective becomes differentiable
(see \Cref{fig:etaProfile}).
To this end, we define $\eta_\delta(s) \coloneqq \eta(\delta^{-1}s)$ using the $C^1$ smooth, cubic Hermite blending function
$\eta\colon\R\to \R_{\geq0}$ with $\eta(s)=1$ for $s<-1$, $\eta(s)=\tfrac14 (s+2)(s-1)^2$ for $|s|\leq 1$ and $\eta(s)=0$ for $s>1$.
For \(\delta \to 0\) and $s\ne 0$, \(\eta_\delta(s)\) converges to \(\chi_{\{s<0\}}\) and thus
\(\eta_\delta(\sdf(\cdot))\) converges pointwise to \(\chi_{[\sdf < 0]}\)
on $[\phi\ne  0]$ for a fixed \(\sdf\). 
Finally, this allows us to introduce the smoothed approximation
\begin{align} 
	\label{eq:normal_loss_delta}
	\normalloss(\sdf) \approx \int_\domain  \eta_\delta \circ \sdf \lvert \nabla \sdf + n \rvert^2 + (1-\eta_\delta \circ \sdf) \lvert \nabla \sdf - n \rvert^2\d x
\end{align}
of the normal approximation objective.
We apply the same smoothing to \(\orientloss\) and obtain
\begin{equation}
	\orientloss(\sdf) \approx \int_{\B^-} 1-\eta_\delta \circ \sdf \d x + \int_{\B^+} \eta_\delta \circ \sdf \d x,
\end{equation}
which promotes that \(\sdf \leq -\delta\) in \(\B^-\) and \(\sdf \geq \delta\) in \(\B^+\).
\begin{figure}[ht]
	\centering
	\includegraphics[width=\linewidth]{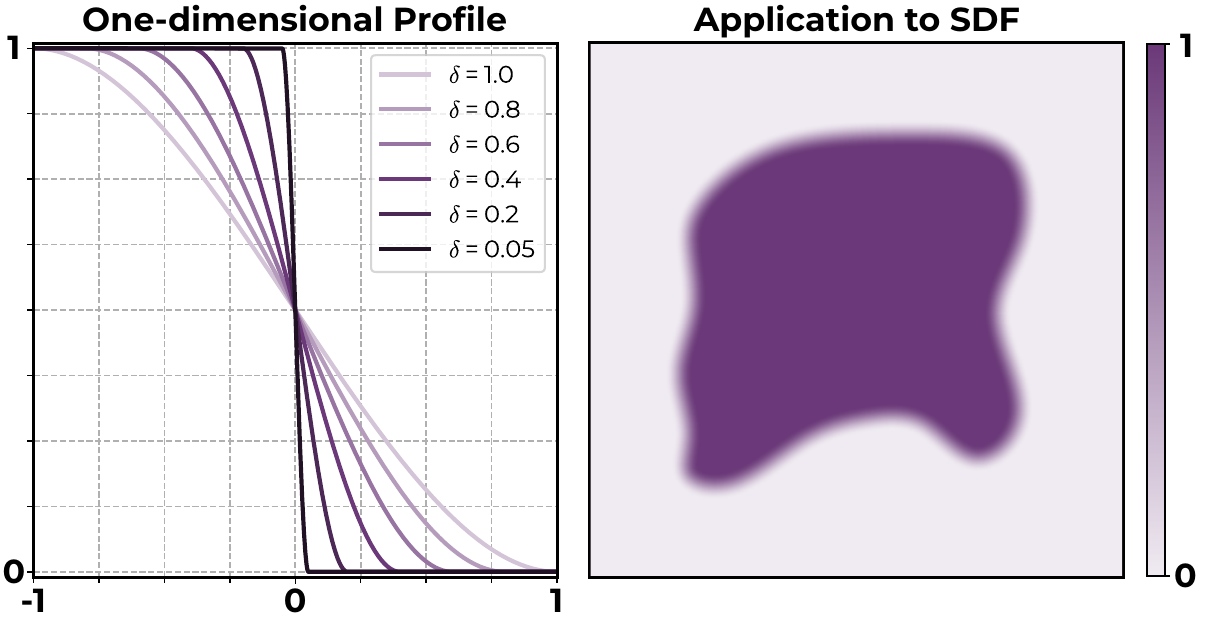}
	\caption{
		Left: sketch of the blending function $s\mapsto \eta_\delta(s)$ for varying $\delta$ value; right: 2D sketch of $\eta_\delta \circ \phi$ with $\delta=0.05$ approximating the characteristic function of the interior of \(\surface\). \label{fig:etaProfile}
	}
\end{figure}
To simplify the numerical integration of the orienting objective \(\orientloss\), we use midpoint quadrature, \ie
\begin{equation}\label{eq:BpmLossreg}
	\orientloss (\sdf)\approx  h^3 \sum_{\substack{\alpha\in \mathbb{Z}^3\\B_\alpha \subset \B^-}} 
	(1 - \eta_\delta(\sdf(\alpha h))) + h^3 \sum_{\substack{\alpha\in \mathbb{Z}^3\\B_\alpha \subset \B^+}} \eta_\delta(\sdf(\alpha h)).
\end{equation}
Finally, we choose \(\delta\) and \(h\) such that $2\delta \leq h$, which ensures that  the support of the integrand in \eqref{eq:BpmLossreg}
does not intersect the surface \(\surface\).

\section{Experimental Results}
We implemented our method using the PyTorch framework. All experiments took under an hour, including both optimization steps. 
All experiments were performed on a workstation with two AMD EPYC 7402 processors, 256GB RAM, and an NVIDIA A100 with 40GB memory.
The code is available at \url{https://github.com/sweidemaier/HeatSDF}.

\subsection{Parameters and Error Metrics}

\paragraph*{Network parameters.}
To represent \(u\) and \(\sdf\), we use fully connected neural networks with periodic activation functions as proposed by \cite{sitzmann2020implicit}.
Concretely, we use four hidden layers of size 256 and the sine function as activation in all experiments.
We use the Adam optimizer \cite{Kingma2014AdamAM} to approximately solve our minimization problems.
All experiments were trained for a total number of 50 epochs, each consisting of 1000 batches of constant batch size. During training, we decrease the learning rate up to $10^{-8}$ using the PyTorch function ReduceLROnPlateau with patience $2$; as an initial learning rate, we choose $10^{-4}$. 

\paragraph*{Weighting of the loss terms.}
We have a single weighting parameter $\lambda_{\mathrm{fit}}$ in front of $\fitloss$, constraining the SDF to be close to zero on the input point cloud. We see a trade-off between a good surface fitting minimizing $\fitloss$ and the normal alignment loss $\normalloss$. This is reflected in the decreasing reconstruction error for increasing $\lambda_{\mathrm{fit}}$ which comes with an increase of the eikonal error, see \Cref{fig:lambdastudy}. Extremely high values of $\lambda_{\mathrm{fit}}$ might lead to surfaces with small artifacts. However, for the range covered in the plot, all surfaces can be recovered without artifacts.
In all our experiments, we choose $\lambda_{\mathrm{fit}} = 100$ as our focus lies on good SDF qualities.
\begin{figure}[t]
	\centering
	\includegraphics[width=\linewidth]{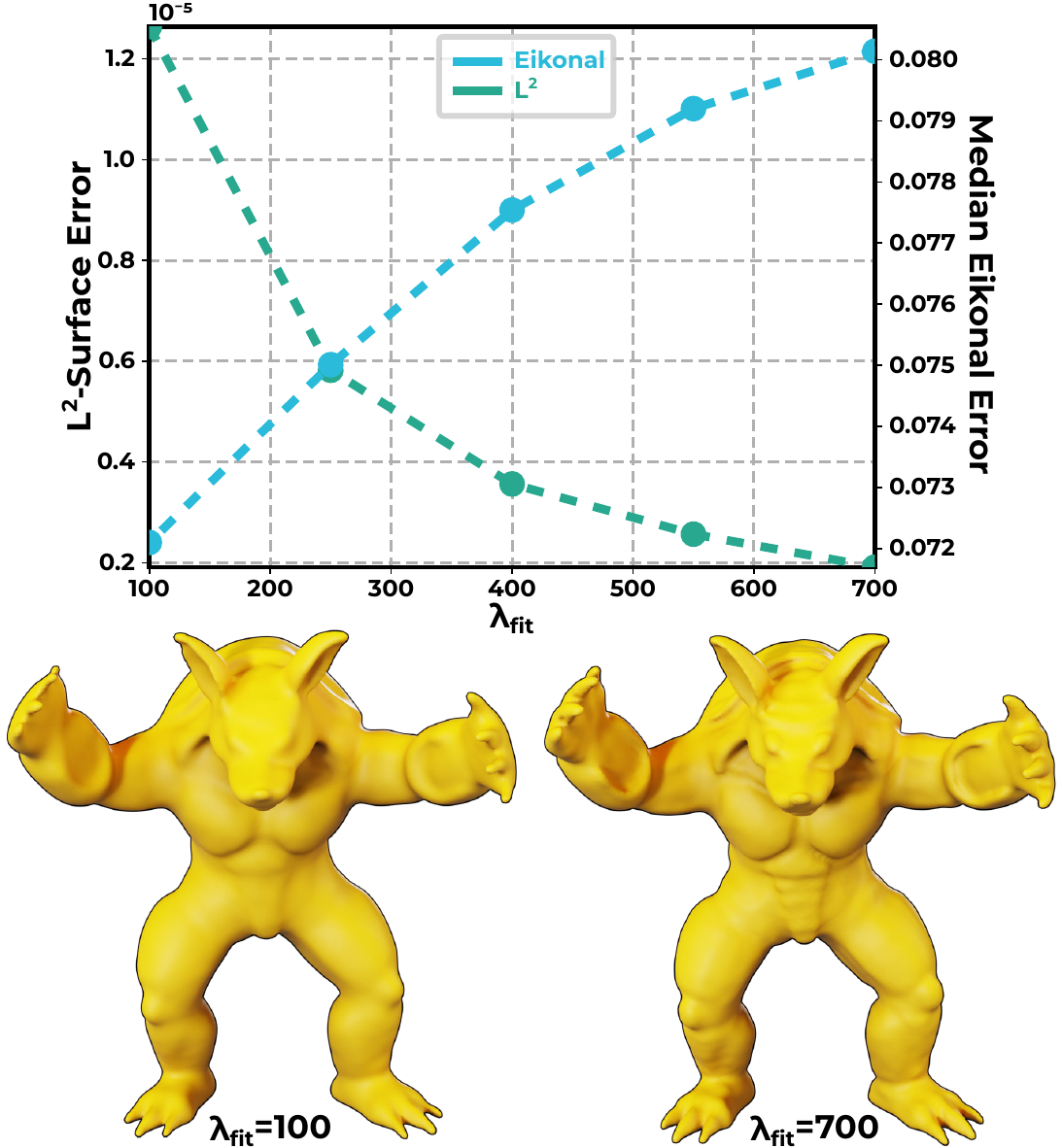}
	\caption{(top) Increasing $\lambda_{\mathrm{fit}}$ leads to a better surface fitting but higher eikonal error for the
		armadillo shape; (bottom) qualitative results for $\lambda_{\mathrm{fit}}= 100$ (left) and $700$ (right). \label{fig:lambdastudy}}
\end{figure}

\paragraph*{Timestep and blending function.}
For all experiments, we choose $\tau=0.005$ for the small time step and $\hat{\tau}=0.1$.
We use a smoothed blending function $\eta_\delta$ to approximate the characteristic functions of the sets \([\sdf < 0]\) and \([\sdf > 0]\) in \eqref{eq:normal_loss_delta}.
We study the dependence of our result on the parameter $\delta$, which corresponds to the steepness of $\eta_\delta$. For $\delta \in [0.001, 0.01]$ the $L^2$ surface error only changes by orders of $10^{-6}$ and the eikonal error in a narrow band by $10^{-4}$, see \Cref{fig:deltastudy}. Hence, the choice of this parameter is not very sensitive. 
In all our experiments, we choose $\delta=0.005$.

\begin{table}
	\centering
	\begin{tabular}{lcc}
		\toprule
		\multicolumn{1}{c}{$\mathbf{\delta}$} &$\mathbf{E_{recon}^{\mathcal{S}}}$ & $\mathbf{E_{eik}}$\\
		\midrule
		0.01&\(1.312\cdot 10^{-5}\)&0.07020\\
		0.0075&\(1.287\cdot 10^{-5}\)&0.07133\\
		0.005&\(1.287\cdot 10^{-5}\)&0.07210\\
		0.0025&\(1.248\cdot 10^{-5}\)&0.07381\\
		0.001&\(1.161\cdot 10^{-5}\) &0.07418 \\
		\bottomrule
	\end{tabular}
	\caption{Error measurements for varying smoothing parameter $\delta$ of the blending function for the armadillo model and $\lambda_{\mathrm{fit}}  =100$.\label{fig:deltastudy}}
\end{table}

\paragraph*{Evaluation metrics.} \label{sec:evaluationMetrics}
We use two metrics defined on the input point cloud, evaluating the SDF and its gradient on the zero-level set, and two metrics defined in a \emph{narrow band} sampled near the surface, evaluating the SDF and the eikonal error.

To measure the reconstruction error, we compute the deviation of the neural SDF $\phi$ from being zero on the ground truth surface in an $L^2$ sense:  
\begin{equation*}
	\mathbf{E_{recon}^{\mathcal{S}}}\coloneqq\frac{1}{\vert \mathcal{M}\vert }\sum_{x \in \mathcal{M}} |\phi(x)|^2
\end{equation*}
using a set $\mathcal{M}$ of $50$k distinct on-surface points sampled on the ground truth surface.
To this end, we exploit that all our considered models come with a ground truth mesh.
Further, we evaluate the normal alignment error on the input surface by computing the cosine distance:
\begin{equation*}
	\mathbf{E_{recon}^n} \coloneqq 1 - \frac{1}{|\mathcal{F}|} \sum_{x\in \mathcal{F}} n(x)\cdot\frac{\nabla  \phi(x)}{\vert\nabla \phi(x)\vert} 
\end{equation*}
where $n(x)$ denote the discrete normals of the respective ground truth mesh at the triangle centers $x$ and $\mathcal{F}$ 
the set of face centers of the respective mesh.

We further evaluate the quality of the results on a set $\mathcal{N}$ of $10$k points in a narrow band with distances in $[-0.1,0.1]$ from $\surface$. 
The set $\mathcal{N}$ was generated once per model using rejection sampling based on the ground truth SDF and was subsequently used consistently across all evaluations.
We compute the SDF error compared to a ground truth signed distance  
$\mathrm{d}(\cdot, \mathcal{S})$ to the mesh surface $\mathcal{S}$:
\begin{equation*}
	\mathbf{E_{SDF}} \coloneqq \frac{1}{|\mathcal{N}|}\sum_{x\in \mathcal{N}}|\phi(x) - \mathrm{d}(x, \mathcal{S})|,
\end{equation*}
as well as the eikonal error: 
\begin{equation*}
	\mathbf{E_{eik}} \coloneqq \mathrm{median} \left\{x\in \mathcal{N} \, \mid\,  |1 - |\nabla\phi(x)||\right\}. 
\end{equation*}

\subsection{Comparisons}
\subsubsection{Comparison to Neural SDF Methods}\label{sec:Comp}
We evaluate our method by comparing it to other neural approaches that compute signed distance functions from point clouds. 

We compare with 1-Lip \cite{coiffier20241} that focuses on having low eikonal errors and good distance approximations by making use of a Lipschitz network architecture. It additionally requires oriented normals.  Moreover, we consider HESS \cite{wang2023neural}, SALD \cite{atzmon2021sald}, and HotSpot \cite{wang2025hotspotsigneddistancefunction}. For details on the models and point clouds, see \Cref{sec:AppB}.

In \Cref{fig:quantComp}, we show a quantitative comparison of the different methods over a range of shapes (\cf \Cref{sec:AppB}, \Cref{fig:OurZeros}). We plot the surface reconstruction error \(\mathbf{E_{recon}^{\mathcal{S}}}\) against the two measures for SDF quality \(\mathbf{E_{eik}}\) and $\mathbf{E_{SDF}}$. The results show that our method achieves a balanced trade-off between an accurate surface reconstruction and a good SDF approximation. 
This is in contrast to HESS, which produces good reconstruction results but has larger eikonal and SDF errors and 1-Lip, which has low eikonal errors but loses detail in the surface reconstruction. Some results of SALD and HESS had flipped signs, which we inverted before measuring the errors.

\begin{figure}[!ht]
	\centering
	\includegraphics[scale=0.5]{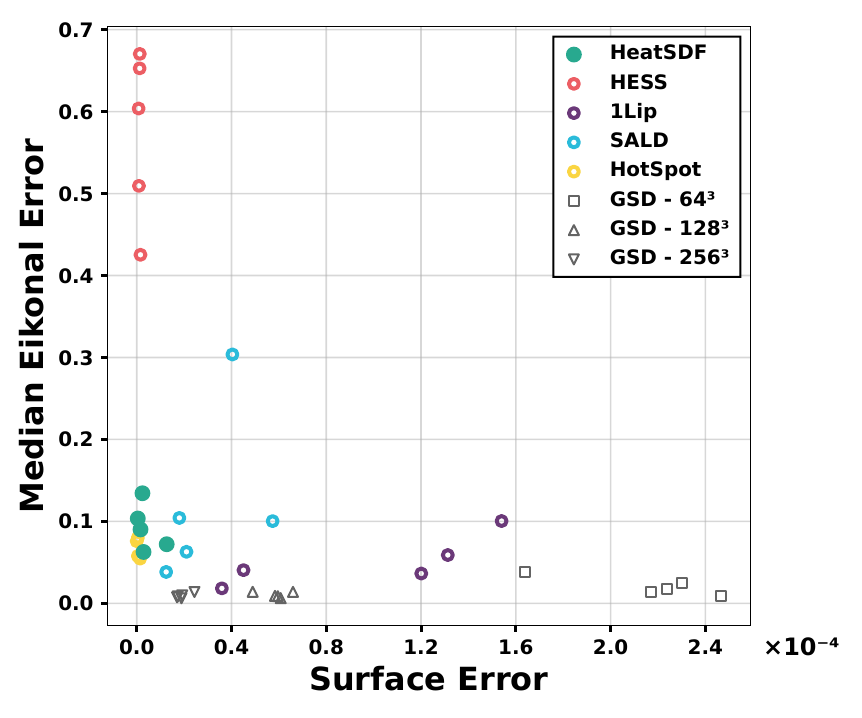}
	
	\vspace{0.0em}
	
	\makebox[0pt][c]{\hspace{-7pt}\includegraphics[scale=0.5]{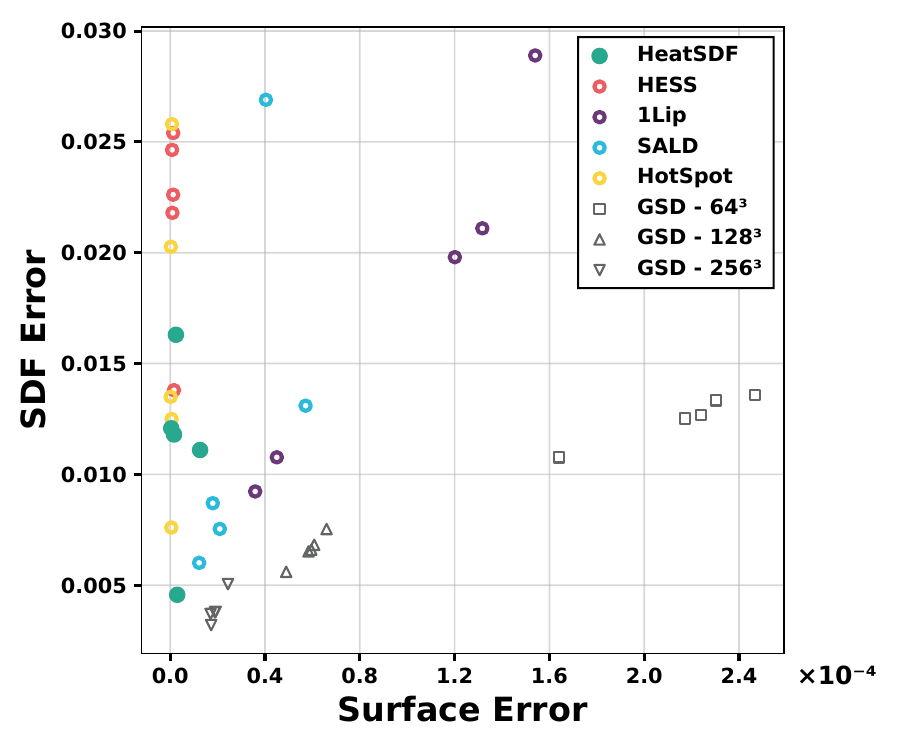}}
	\caption{Quantitative evaluation for different methods applied to various models (\cf \Cref{sec:AppB}, \Cref{fig:OurZeros}), illustrating
		that our method effectively balances surface reconstruction and SDF approximation; top:
		eikonal error $\mathbf{E_{eik}}$ against surface error $\mathbf{E_{recon}^{\mathcal{S}}}$; bottom: 
		SDF error $\mathbf{E_{SDF}}$ against surface error $\mathbf{E_{recon}^{\mathcal{S}}}$.
		\label{fig:quantComp}}
\end{figure}

\Cref{fig:qualComp} shows a qualitative comparison between the zero-level sets extracted with Marching Cubes for the different methods. Moreover, we plot for points, sampled in a narrow band near the surface, the neural distances compared the ground truth signed distance to the mesh. While HESS, HotSpot, and our method all achieve an accurate shape reconstruction, our method leads to the most accurate SDF estimation. 
We also note that HESS has a very accurate SDF estimation only very close to the surface and $1$-Lip consistently underestimates the distance, as expected from these methods. Those observations are also reflected in the corresponding table.

Our quadrature for the surface integral \Cref{eq:weights} accounting for the input point cloud density allows us to compute accurate SDFs even for spatially varying point densities. To demonstrate this, we consider a capped torus shape that can be described as the zero-level of a closed-form SDF \cite{iquilez2025sdf}. We consider a dense, a sparse, and a non-uniform point cloud as well as an example with noise added to the surface; see \Cref{tab:beanexperiment} for a visualization of the point clouds and quantitative values comparing the uniform and non-uniform case. All methods can deal with noisy and sparse input data. We provide the quantitative evaluation in the appendix \Cref{tab:appQuantTorus}. However, HotSpot and SALD have difficulties dealing with the highly non-uniform point cloud, mixing inside and outside of the shape. In contrast, our SDF error only changes by a magnitude of $10^{-4}$. In \Cref{fig:bean}, we show scatter plots for the SDF evaluation on the non-uniform capped torus.
\begin{table}
	\centering
	\footnotesize
	\includegraphics[scale=0.036]{figureTable2.pdf}
	\begin{tabular}{cccccc}
		\toprule
		&\textbf{Method} &$\mathbf{E_{recon}^{\mathcal{S}}}$ & $\mathbf{E_{eik}}$ &$\mathbf{E_{recon}^n}$ & $\mathbf{E_{SDF}}$\\
		\midrule
		
		\multirow{5}{*}{\rotatebox{90}{\tiny\textbf{Uniform}}}
		&SALD&  $5.27\cdot10^{-5}$ & 0.07901  & 0.02328&0.00929\\
		& HESS&\(2.19\cdot10^{-7}\)& 0.73540  & \textbf{0.00004}&0.01866\\  
		& 1-Lip & \(1.32\cdot10^{-5}\) & \textbf{ 0.00684 } & 0.00131&0.00339\\
		&HotSpot& \(\mathbf{1.72\cdot10^{-7}}\)&  0.03958  & 0.00007 & 0.00699\\ 
		&Ours & \(3.85\cdot10^{-7}\)&    0.03529  & 0.00079 &\textbf{0.00210}\\  
		\midrule \multirow{5}{*}{\rotatebox{90}{\tiny\textbf{Non-Uniform\hspace{0.05cm}}}}
		&SALD& \(3.62\cdot10^{-4}\) &  0.26120 & 0.08206&0.02397\\ 
		& HESS& \(\mathbf{2.21\cdot10^{-7}}\)&  0.73540 &  \textbf{0.00004}&0.01866\\  
		&  1-Lip  & \(2.85\cdot10^{-5}\)&  \textbf{0.00491} & 0.00191&0.00447\\  
		& HotSpot & \(2.16\cdot10^{-6}\)&   0.03669 &  0.00044 & 0.01322\\  
		& Ours  & \(7.89\cdot10^{-7}\)&    0.03682 &  0.00105 &\textbf{0.00221}\\  
		\bottomrule
	\end{tabular}
	\caption{Different point cloud sampling variants for the same capped torus shape with a given ground truth SDF (top), and a quantitative comparison of  uniform and  non-uniform sampling (bottom). Our locally adaptive weighting scheme enables  accurate SDF approximation even in the presence of non-uniform sampling density. }
	\label{tab:beanexperiment}
\end{table}

\begin{figure*}[ht]
	\begin{minipage}{0.5\textwidth}
		\includegraphics[scale=0.074]{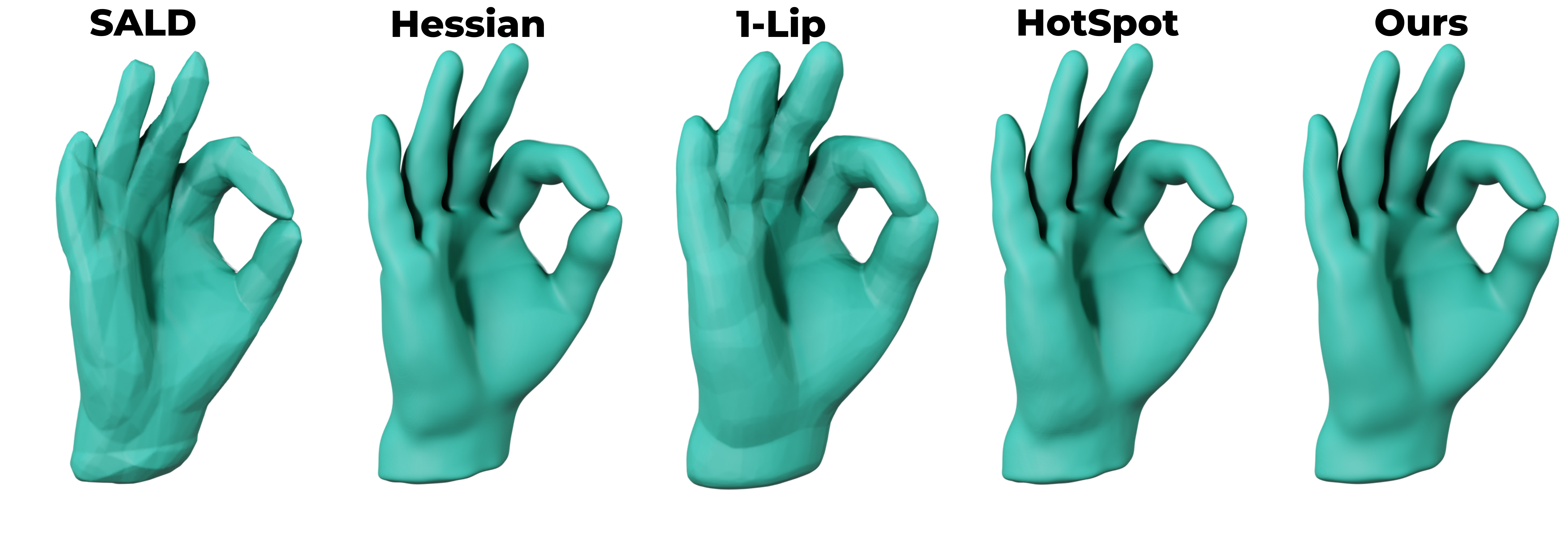}
	\end{minipage}
	\begin{minipage}[t]{0.42\textwidth}
		\footnotesize
		\begin{tabular}{ccccc}
			\toprule
			\textbf{Method} &$\mathbf{E_{recon}^{\mathcal{S}}} $ & $\mathbf{E_{eik}}$ &$\mathbf{E_{recon}^n}$ & $\mathbf{E_{SDF}}$\\ 
			\midrule
			SALD & \(5.74\cdot 10^{-5}\)&  0.1003 & 0.10630 & 0.01308\\ 
			HESS & \(1.27\cdot 10^{-6}\)  &  0.6704 &\textbf{0.00134} &0.02547   \\ 
			1-Lip   & \(1.20\cdot 10^{-4}\)   & \textbf{0.0363} &0.03793&0.01981\\  
			HotSpot  &\(\mathbf{7.49\cdot 10^{-7}}\) &   0.0651 & 0.00152 &  0.02578\\  
			Ours & \(1.59\cdot 10^{-6}\)  & 0.0901 &  0.00262&\textbf{0.01184}\\  %
			\hline
			\bottomrule
		\end{tabular}
	\end{minipage}
	\centering
	\includegraphics[scale=0.32]{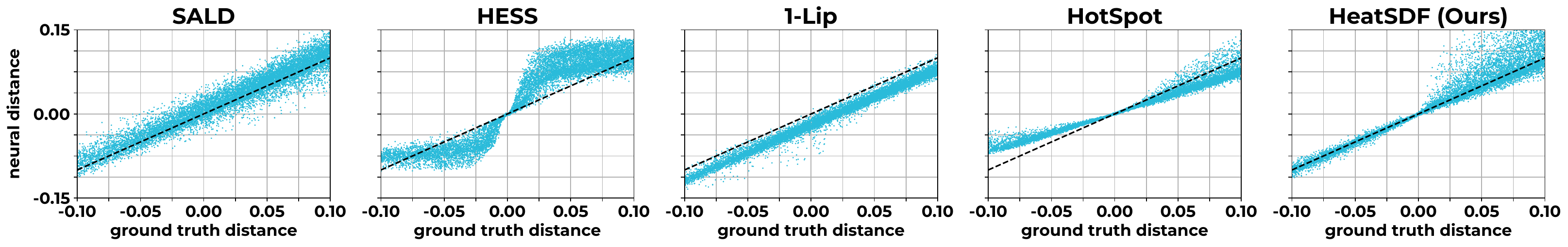}
	\caption{Results for the hand model from 6k-point cloud.  The zero-level set of HESS, HotSpot, and
		HeatSDF (present results) are qualitatively hard to distinguish. However, in the error
		table our model shows a favorable SDF. Furthermore, scatter plots show the different characteristics of the error profiles:
		high consistency close to the surface but larger and spread deviation farther away (HESS), globally good consistency with a tendency to underestimate the distance (1-Lip), good consistency with a slight difference in the slope (HotSpot), and overall consistency for our method.
		\label{fig:qualComp}}
\end{figure*}

\begin{figure*}[ht]
	\centering
	\includegraphics[scale=0.32]{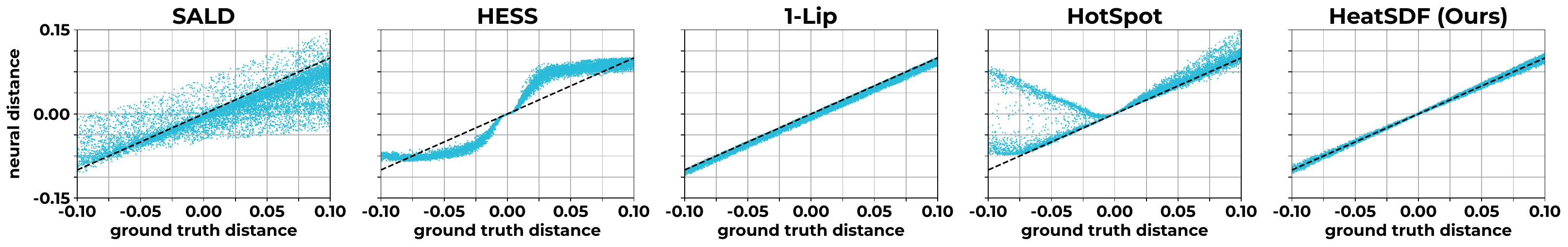}
	\caption{SDF error scatter plots for the non-uniformly sampled capped torus from \Cref{tab:beanexperiment}. Only 1-Lip and our method show a globally high consistency of the SDF. All other methods
		significantly deviate from the ground truth in the narrow band.}
	\label{fig:bean}
\end{figure*}

\subsubsection{Comparison to Grid-based Methods}
\label{sec:grid_comparison}
\begin{figure*}[ht]
	\centering
	\begin{minipage}[c]{0.5\textwidth}
		
		\includegraphics[scale=0.35]{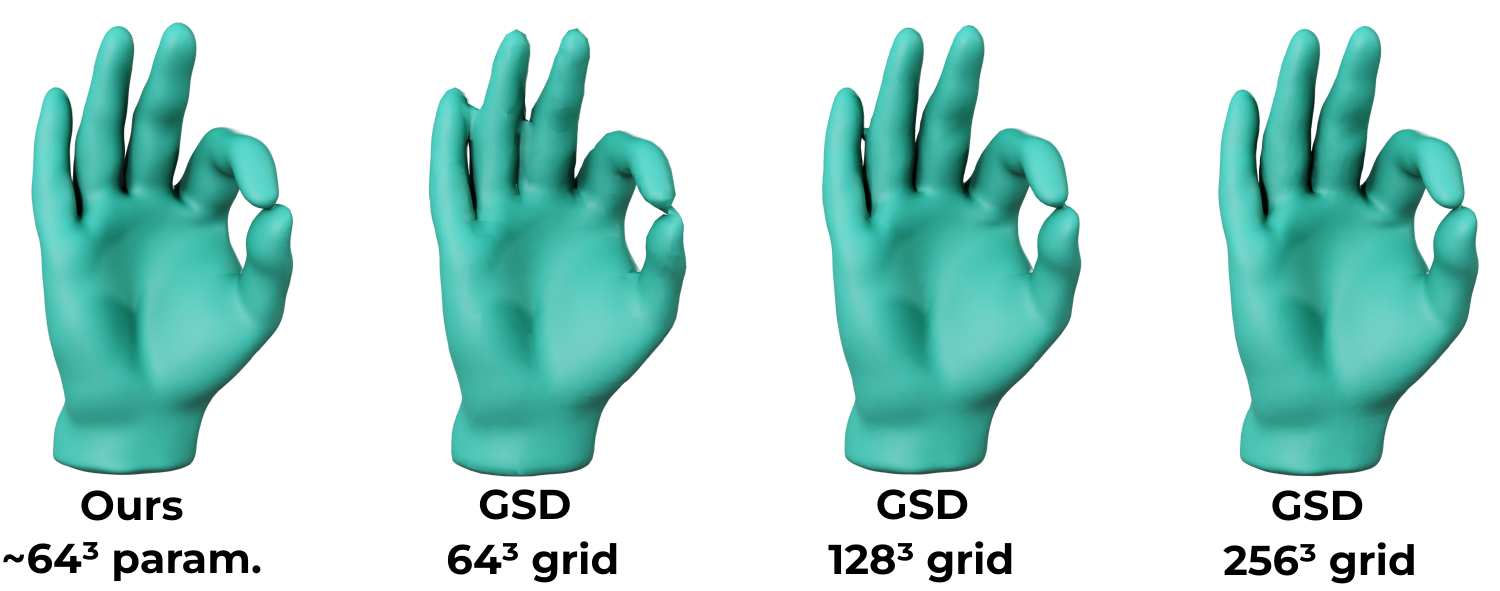}
		
	\end{minipage}\hfill
	\begin{minipage}[t]{0.5\textwidth}
		
		\centering
		\begin{tabular}{cccccc}
			\toprule
			&\textbf{DoF} & $\mathbf{E_{recon}^{\mathcal{S}}}$ & $\mathbf{E_{eik}}$ & $\mathbf{E_{recon}^n}$ & $\mathbf{E_{SDF}}$ \\
			\midrule
			\multirow{3}{*}{\rotatebox{90}{GSD}}& $64^3$ & $2.24\cdot 10^{-4}$ & 0.0174 & 0.04005 & 0.01268 \\
			& $128^3$ & $5.95\cdot 10^{-5}$ & 0.0087 & 0.01289 & 0.00661 \\
			&$256^3$ & $1.88\cdot 10^{-5}$ & \textbf{0.0072} & 0.00497 & \textbf{0.00371} \\
			\midrule
			Ours& $\sim64^3$ & $\mathbf{1.59\cdot 10^{-6}}$ & 0.0901 & \textbf{0.00262} & 0.01184 \\
			\bottomrule
		\end{tabular}
	\end{minipage}
	\caption{
		{Results for the same hand model as in \Cref{fig:qualComp}. 
			Our neural method outperforms the GSD approach~\cite{feng2024heat} in terms of surface $L^2$ error (${E_{recon}^{\mathcal{S}}}$) and surface normal alignment (${E_{recon}^n}$) across all resolutions. 
			Moreover, for a comparable number of degrees of freedom, our method achieves a similar accuracy with respect to the distance error (${E_{SDF}}$).    \label{fig:gridComp}}}
\end{figure*}

\begin{figure}[h]
	\centering
	\includegraphics[width=\columnwidth]{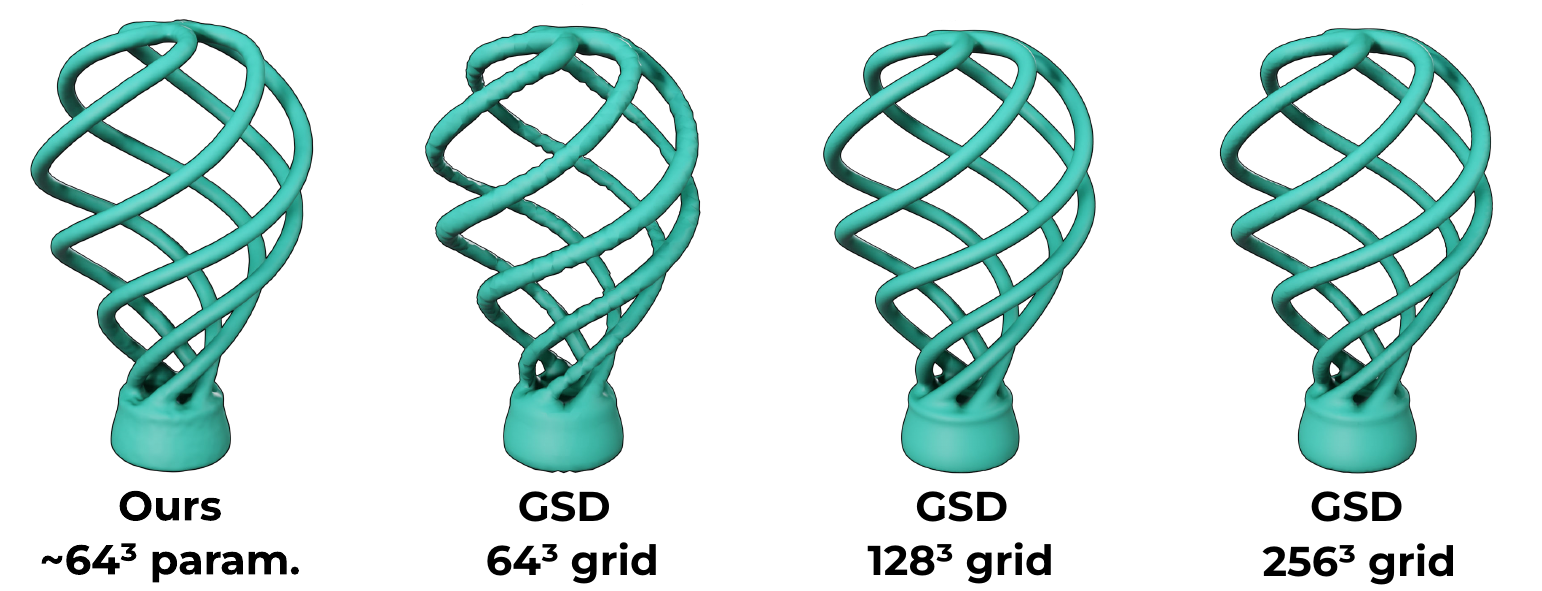}
	
	\caption{Comparison of our method with GSD on a lightbulb shape. Results are shown for GSD with approximately the same number of degrees of freedom (second from left) and for higher resolutions (two right columns), compared to our result (left).
		\label{fig:FC24_light}}
\end{figure}

In addition to the comparisons with state-of-the-art neural approaches presented in \Cref{sec:Comp}, 
we also compare our method to the closest non-neural, heat-based approach \emph{Generalized Signed Distances} (GSD) \cite{feng2024heat}. 
In GSD, the first step of the classical heat method is modified, such that the input geometry's normal vectors are diffused using a short-time heat flow. 
The second step, namely normalizing those diffused vectors and integrating the resulting signed normal field to recover the signed distance function, remains unchanged. 
Unlike \cite{feng2024heat}, we do not diffuse the normal field. 
Instead, we modify the second step of the heat method to obtain signed distances from an unsigned normal field.
Nevertheless, both methods share the use of the heat flow to recover a continuous approximation of the normals of the distance function.

\Cref{fig:gridComp} shows a qualitative comparison between the zero-level sets for our neural method and that for GSD at three different grid resolutions, along with a quantitative evaluation. 
The grid-based GSD method with approximately the same number of degrees of freedom as our neural network exhibits notable artifacts in regions of near self-contact (e.g., merged fingers). Additionally, it struggles in regions of high curvature (e.g., the thin wires of the lightbulb shape), as illustrated in \Cref{fig:FC24_light}.
Achieving a reconstruction quality comparable to ours requires a significantly higher grid resolution, \ie more degrees of freedom. 
This is also reflected by the quantitative results in \Cref{fig:quantComp}.
Our method requires on average approximately $25$ minutes to train, whereas the grid-based method takes, on average across the presented models, $1.5$ minutes ($64^3$), $10$ minutes ($128^3$), and $75$ minutes ($256^3$) to compute a solution. 
For these comparisons, we used the open source code provided by the authors of \cite{feng2024heat} with some straightforward parallelizations.
In both cases, runtimes heavily depend on the size of the input point cloud.

\subsection{Application: Neural PDE Heat Flow on Neural Surfaces}
In this section, we demonstrate that our learned SDF is sufficiently accurate to solve a PDE directly on the neural field using a level set method.
Following~\cite{bertalmio2001variational,dziuk2008eulerian}, we start by deriving a level set method for the geometric heat equation and formulate it as a minimizing movement scheme to make it accessible for a neural network-based approach (cf.~\cite{park2023deep}). 

Let $\phi$ be the SDF of a surface $\surface$ and consider the set of parallel surfaces
$\surface_c \coloneqq \{ x \in \domain \mid \phi(x) = c \}$ 
for different values $c$ of the level set function,  with $\surface = \surface_0$.
Then $n(x)=\nabla \phi(x)$ is the normal at the point $x$ on $\surface_{\phi(x)}$ 
and  $P(x) = \text{Id} - n(x) \otimes n(x)$  the
projection onto the tangent space $T_x \surface_{\phi(x)}$.
For the tangential gradient of a function $w$ one obtains $\nabla_{\surface_{\phi(x)}} w(x) = (P \nabla w)(x)$.  

A central tool in level set calculus is  the coarea formula for the integration of a function $f$, which reads as
$
\int_\R  \int_{\surface_c} f(x) \d a \d c = \int_{\domain} f(x) \d x
$
for an SDF representation of the level sets.
With the coarea formula and the tangential gradient at hand one obtains for the weak formulation 
of the geometric heat equation $\partial_t w - \Delta w = 0$ with square integrable initial data $w(0,\cdot) = w^0$
on the level sets $\surface_c$:
\begin{align}
	0=\int_{\domain} \left(\partial_t w\,\vartheta   + P \nabla w \cdot \nabla \vartheta \right)  \d x
\end{align}
for all smooth test functions $\vartheta \colon \R^3 \to \R$ and all $t>0$.
Using integration by parts in the second term, we get the strong form
$\Delta_{\surface} w = \div (\nabla w \!-\! (n\cdot \nabla w) n)$ of the
Laplace-Beltrami operator on $\surface$.  

The heat equation can be discretized in time by a 
minimizing movement formulation of a fully implicit Euler scheme~\cite{park2023deep} for the geometric heat equation, 
where for $k=0,\ldots$ one iteratively defines $w^{k+1}$ as the minimizer 
of the functional
\begin{align}
	\mmheat[w]&\coloneqq \int_{\domain} (w-w^k)^2  + \tau \left\vert P \nabla w  \right\vert^2 \d x \nonumber\\
	&=\int_{\domain} (w-w^k)^2  + \tau \left( \vert \nabla w \vert^2 - \vert \nabla \phi \cdot \nabla w\vert^2 \right) \d x
	\label{eq:narrowband}
\end{align}
with $w^{k+1}$ approximating $w((k+1)\tau,\cdot)$ for the time step size $\tau$ (cf. \eqref{eq:MMsimple} for the Euclidean case).

The minimization problem in \eqref{eq:narrowband} decouples over all level sets $\surface_c$.
This is only guaranteed for this continuous formulation. A numerical approximation
comes with only an approximate decoupling. In practice, we confine to a narrow band using a smooth blending function 
$\mu_\sigma$ with $\mu_\sigma(s)= 1$ on $[-\tfrac\sigma2,\tfrac\sigma2]$ and vanishes outside $(-\sigma,\sigma)$, where in our application $\sigma=0.05$.

Altogether, a time-discrete, fully implicit Euler scheme can be formulated as a variational problem and thus
solved numerically using a neural network approach. \Cref{fig:heat_on_surface} (bottom) shows the results of neural heat diffusion on the neural zero-level set of the Max Planck head, in comparison to a
piecewise affine, continuous finite element method on the original triangulation (top) from which the point cloud was extracted. Additionally, we present results of the finite element heat flow computation on the mesh extracted from the neural zero-level set (middle). We observe that the three approaches lead to very similar results qualitatively, indicating that our neural-PDE-on-neural-surface framework yields valid results \emph{without} requiring to extract the mesh of the neural level set.

\begin{figure}[!ht]
	\centering
	\includegraphics[width=\linewidth]{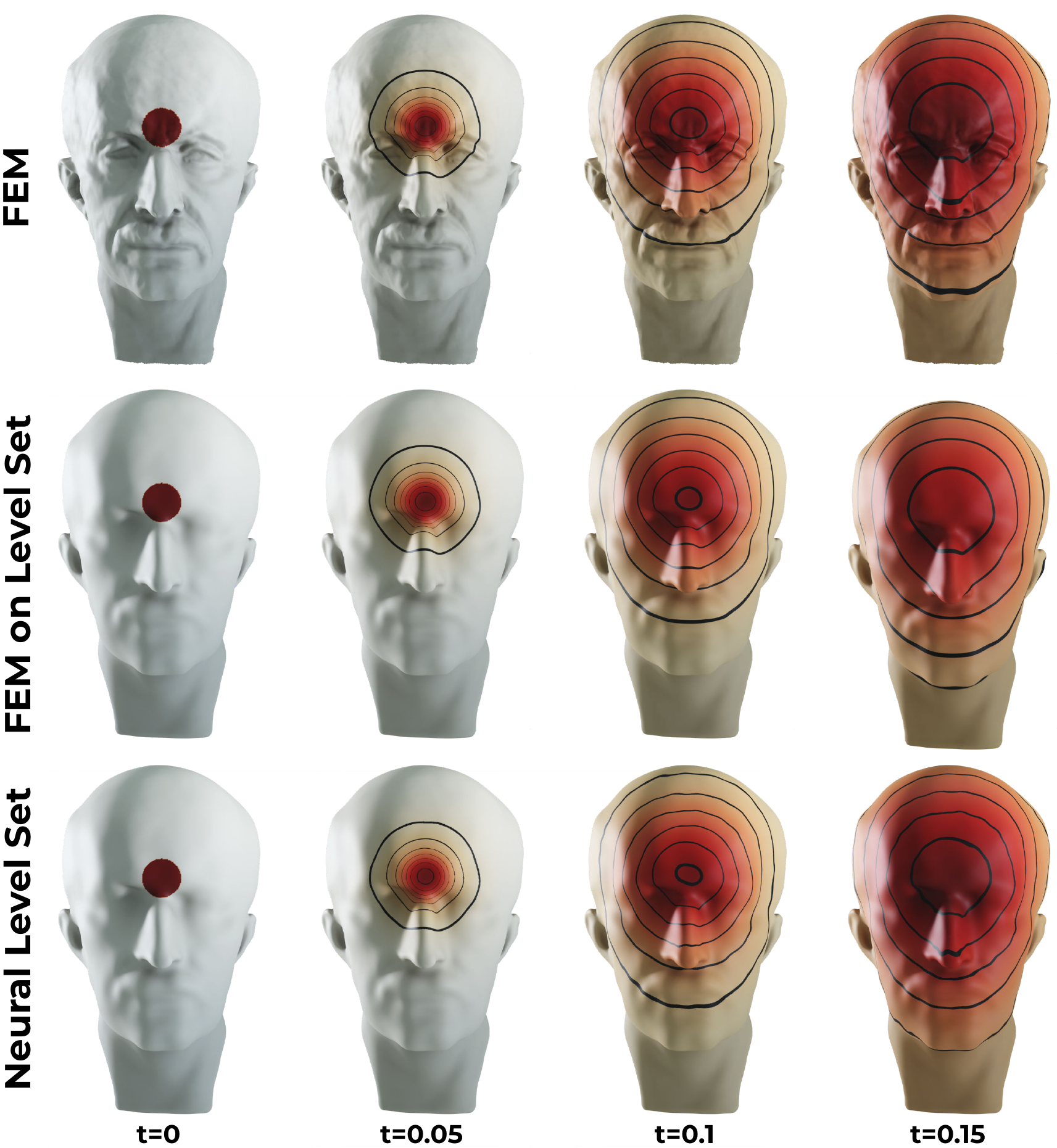}
	\caption{Level set curves for different time steps of the heat diffusion with a characteristic function of a ball intersecting the surface as initial data: 
		neural heat flow on neural zero-level set (bottom), finite element backward Euler scheme on a triangulation of the input point cloud (top) and on the regularized marching cube triangulation of the neural zero-level set (middle). We note that the neural heat flow on the neural surface is qualitatively similar to the FEM solutions.}
	\label{fig:heat_on_surface}
\end{figure}

\subsection{Application: Geometric Queries}
\begin{figure}
	\centering
	\includegraphics[width=\linewidth]{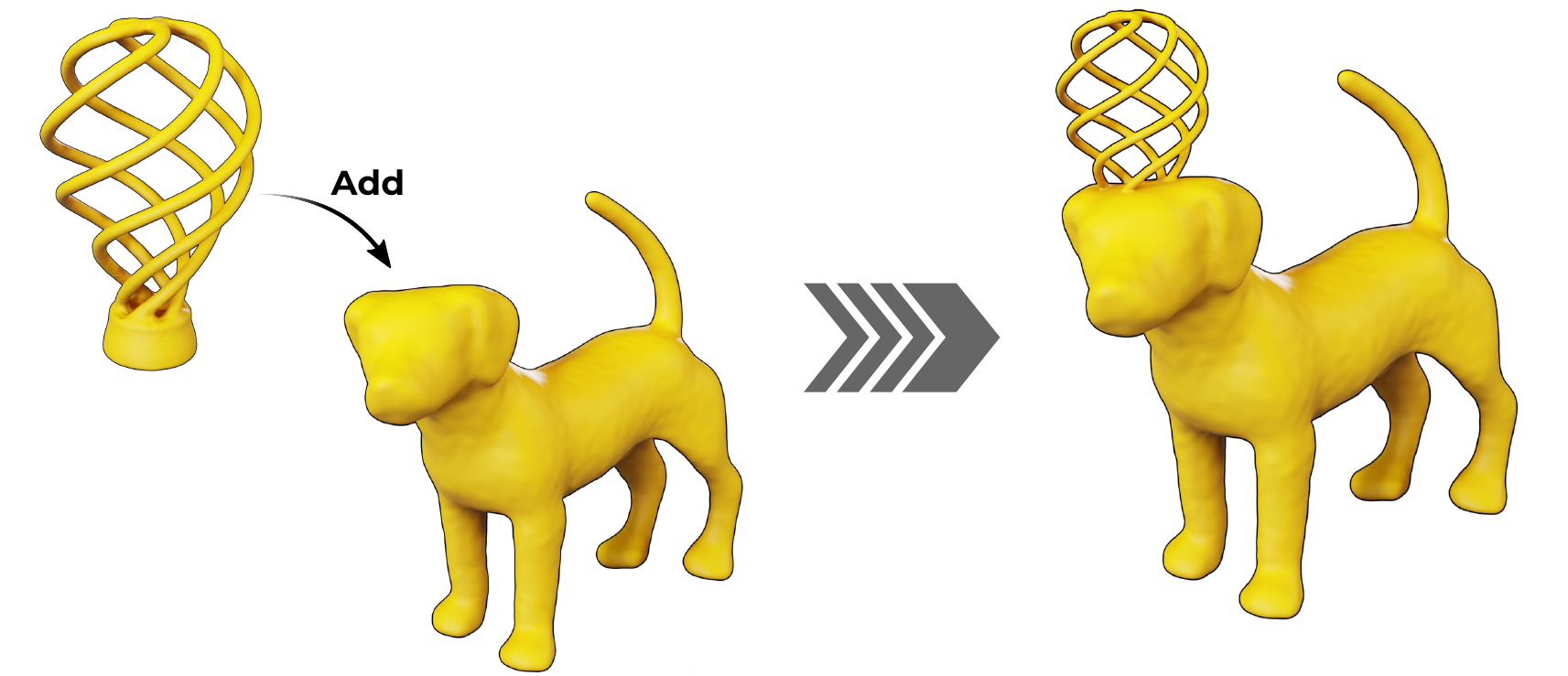}
	\caption{The union of the light bulb and the dog volume can be obtained as the sublevel set of the pointwise minimum of the corresponding SDFs.
	}
	\label{fig:union_lightbulb_dog}
\end{figure}
We demonstrate that our approach is accurate enough to compute the union or intersection of the sublevel sets $[\phi<0]$  for given SDFs. 
For two sublevel sets $[\phi_1<0]$ and $[\phi_2<0]$ we have
\begin{align*}
	[\phi_1<0] \cup [\phi_2<0] &= \left[\min \{ \phi_1,\phi_2\} <0\right] ,\\
	[\phi_1<0] \cap [\phi_2<0] &= \left[\max \{ \phi_1,\phi_2\} <0\right].
\end{align*}
\Cref{fig:union_lightbulb_dog} shows an example result.  
Note that in general, neither $\min \{ \phi_1,\phi_2\}$ nor $\max \{ \phi_1,\phi_2\}$ is an SDF \cite{marschner2023constructive}.
\section{Limitations}
Our method has a few limitations. 
To ensure that we solve only first-order variational problems that are convex
with respect to the function gradients, we train two separate networks, one for the heat diffusion time step $u^\tau$, and one for the SDF $\phi$.
We also observe a moderate loss of detail, which is a common drawback of a neural representations. Additionally, singularities such as crease lines lead to locally larger errors in the SDF. The proposed method for computing the sets \(\B^\pm\) does not generalize to shapes with interior voids or disconnected internal structures, albeit such cases are relatively uncommon. For these challenging cases, our method can be extended by using Dipole Normal Propagation (DNP) \cite{metzer2021orienting} for point cloud orientation and then generalized winding numbers \cite{jacobson2013robust} to determine \(\B^\pm\). In \Cref{fig:armadillo_with_holes}, we show a failure case of our box algorithm, while the extended pipeline succeeds---at the cost of increased computation time.
\section{Conclusion and Future Work}
\begin{figure}
	\centering
	\includegraphics[width=\linewidth]{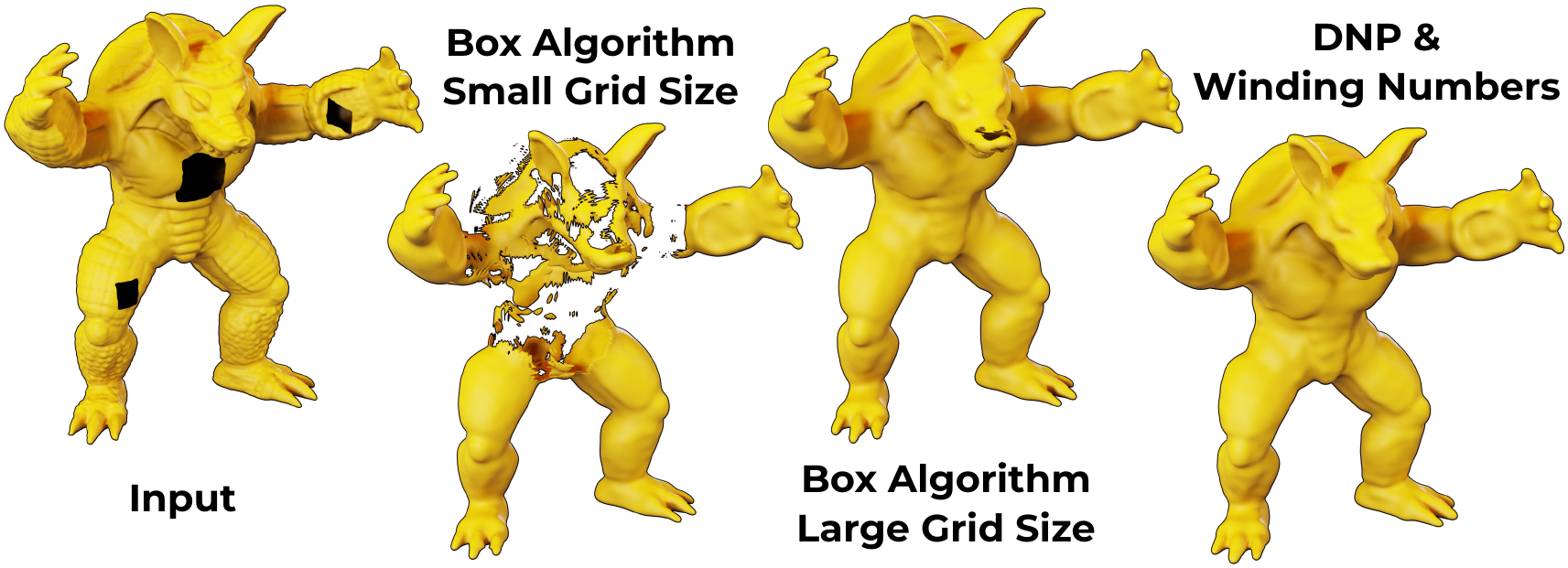}
	\caption{Input point cloud with three holes of varying sizes rendered as a mesh (left) and the extracted zero-level sets using our method with the simple box algorithm with standard grid size (2nd from left), using a four times larger grid size (3rd from left), and using DNP \cite{metzer2021orienting} and generalized winding numbers \cite{jacobson2013robust} to compute the inside/outside separation. The standard grid size results in an unsigned distance function in the corpus region, as it is smaller than the large hole. The larger grid size causes incorrect orientation in small regions, such as the nose, due to a lack of inside points.
	}
	\label{fig:armadillo_with_holes}
\end{figure}
We presented a novel method for computing a 
signed distance function (SDF) of  
a surface represented by an unoriented point cloud. 
Typically, the implicit representation of a surface as an SDF is governed by the eikonal equation,
a transport-type PDE. In the context of neural networks,
this equation is often approached through a variational formulation.
However, the commonly used eikonal loss functional is inherently non-convex, which can hinder convergence.

To address this, our method introduces a
two-step scheme based on two well-posed variational problems, 
offering a more robust and theoretically sound alternative.
First, a time step of the heat flow is computed using an approximate 
area measure as initial data. Second, an $L^2$ fitting of the SDF gradient is performed, using a properly reoriented and normalized gradient derived from the first step.
This two-stage process results in a robust and effective method for computing SDFs. In experiments, it outperforms 
state-of-the-art methods in terms of the error in the signed distance function, making it particularly well-suited as a foundation for level set methods used to solve PDEs on surfaces. 

Future research directions  include the development of narrow band adapted sampling strategies to reduce training times,
optimization of the network architecture, and local resolution enhancement of the SDF near crease lines.
Additionally, extending the method to handle more complex surface PDEs, such as thin film flow or shell deformations, remains a challenging and promising avenue.

\section*{Acknowledgments}
This work was supported by the Deutsche Forschungsgemeinschaft (DFG, German Research Foundation) via project 211504053 -- Collaborative Research Center 1060, 
project  539309657 -- Collaborative Research Center 1720, 
and via Germany’s Excellence Strategy project 390685813 -- Hausdorff Center for Mathematics.
Furthermore, this project has received funding from the European Union’s Horizon 2020 research and innovation program under the Marie Skłodowska-Curie grant agreement No 101034255, as well as the Israel Science Foundation (grant No. 1073/21).
\bibliographystyle{eg-alpha-doi} 
\bibliography{heatSDF.bib}

\appendix\section{Proofs}
\label{appproofs}

Here we give the proofs of \Cref{prop:ex} and \Cref{thm:sdl}.\smallskip

\begin{proof} [\Cref{prop:ex}]
	By the trace theorem
	$u\mapsto \fint_S u \d a$
	is a bounded linear functional on $H^{1}(\domain)$. The quadratic form
	$u\mapsto \int_{\domain} u^2 + \tau \vert \nabla u  \vert^2  \d x$ is $H^{1}(\domain)$-coercive and bounded.
	Thus, by the Lax-Milgram theorem, there is a unique weak solution $u^\tau$ of the Euler-Lagrange equation
	$0 = \partial_u \mmlossmod(u)$ \eqref{eq:MMEL} and this solution is the minimizer of the energy $\mmloss$ \eqref{eq:minimizing_movements_surface}.
\end{proof}

The following lemma will be used in the proof of \Cref{thm:sdl}.

\begin{lemma}\label{lemma:weakcont}
	For $n\in L^2(\domain;\R^3)$, $\phi\in H^1(\domain)$ define $A(\phi,n)\coloneqq\int_\domain \nabla |\phi|\cdot n\d x$. For any fixed $n$, the map $A(\cdot,n)$ is weakly-$H^1$ continuous.
\end{lemma}
\begin{proof}
	Let $\phi_j$ converge weakly to $\phi$ in $H^1(\domain)$
	and consider a sequence $n_k\in C^\infty_c(\domain;\R^3)$ which converges to $n$ strongly in $L^2(\domain;\R^3)$.
	Integration by parts shows that for any $k$ and any $\psi\in H^1(\domain)$ one has
	\begin{align*}
		A(\psi,n_k)=\int_\domain \nabla |\psi|\cdot n_k \d x =
		-\int_\domain |\psi|\div n_k \d x.
	\end{align*}
	Since $\div n_k\in L^2(\domain)$, this implies that $A(\cdot, n_k) \colon L^2(\domain) \to \R$ is continuous  
	and thus by Rellich's theorem 
	$\lim_{j\to\infty} A(\phi_j, n_k)=A(\phi,n_k)$ for any $k\in\N$. 
	Since weak convergence implies boundedness,
	there is $C>0$ such that for any $j$ and $k$
	\begin{align*}
		& |A(\phi_j, n)-A(\phi,n)| \\
		&\le |A(\phi_j, n_k)-A(\phi,n_k)|\\
		& \qquad +|A(\phi_j, n_k-n)|+|A(\phi,n_k-n)|\\
		&\le |A(\phi_j, n_k)-A(\phi,n_k)| \\
		& \qquad + (\|\nabla \phi_j\|_{L^2}
		+\|\nabla \phi\|_{L^2}) \|n_k-n\|_{L^2}\\
		&\le |A(\phi_j, n_k)-A(\phi,n_k)| + C \|n_k-n\|_{L^2}.
	\end{align*}
	Therefore, $\limsup_{j\to\infty} |A(\phi_j, n)-A(\phi,n)|\le C \|n_k-n\|_{L^2}$
	for any $k\in\N$,  and taking $k\to\infty$ we conclude
	$|A(\phi_j, n)-A(\phi,n)|\to0$.
\end{proof}

\begin{proof} [\Cref{thm:sdl}]
	Using $\sgn \phi \, \nabla \phi = \nabla \vert \phi\vert$
	we rewrite $\normalloss$ as
	\begin{align*}
		\normalloss[\phi]  &= \int_\domain\vert \sgn \phi \, \nabla \phi  - n^\tau |^2\d x
		\\
		&= \int_\domain |\nabla \phi |^2 - 2 \sgn \phi \, \nabla \phi \cdot n^\tau + \vert n^\tau\vert^2 \d x \\
		&=  \int_\domain |\nabla \phi |^2 - 2 \nabla \vert \phi\vert \cdot n^\tau + 1 \d x  .
	\end{align*}
	The first term on the right-hand side is weakly lower semicontinuous on $H^{1}(\domain)$.
	By \Cref{lemma:weakcont}  the functional
	$\phi \mapsto A(\phi,n^\tau)$ is weakly continuous on $H^{1}(\domain)$.
	Therefore, $\normalloss$ is weakly lower semicontinuous.

	In turn, by the trace theorem
	for the surface $\surface$ one gets that  $\fitloss$ is weakly continuous on $H^{1}(\domain)$.
	
	Next, we consider $\orientloss$.
	Possibly passing to a subsequence, we can assume that $\phi_j\to\phi$ pointwise almost everywhere. 
	As the function $t\mapsto \chi_{(-\infty,0)}(t)$ is lower semicontinuous,
	for almost every $x$
	we have $\chi_{[\phi<0]}(x)\le \liminf_{j\to\infty} \chi_{[\phi_j<0]}(x)$. By Fatou's Lemma, we conclude that
	$ \int_{\B^-} \chi_{[\phi<0]}\d x\le \liminf_{j\to\infty} \int_{\B^-} \chi_{[\phi_j<0]}\d x.$
	The term with $\B^+$ is treated analogously.
	Altogether, the energy $\sdfloss$ is sequentially weakly lower semicontinuous on $H^{1}(\domain)$.
	
	Furthermore, using $a^2 \leq 2(a-b)^2 + 2 b^2$ with $a=\nabla \vert \phi\vert$ and $b=n^\tau$
	we obtain that
	\begin{align*}
		\Vert \nabla \phi \Vert_{L^2(\domain)}^2
		+ \lambda_{\mathrm{fit}}
		\Vert \phi \Vert^2_{L^2(\surface)}
		&\leq 2 \sdfloss[\phi] + 2 \vert \domain \vert .
	\end{align*}
	Taking into account that $\Vert \nabla \phi \Vert_{L^2(\domain)} + \Vert \phi \Vert_{L^2(\surface)}$ is a norm equivalent to the $H^{1}(\domain)$ norm,
	we get that minimizing sequences are bounded in  $H^{1}(\domain)$. This boundedness together with the
	weak lower semicontinuity ensures the existence of a minimizer of the energy $\sdfloss$.
\end{proof}

\section{Additional Experimental Details} \label{sec:AppB}

Here, we provide some additional qualitative and quantitative evaluation and details on the used data.
In \Cref{tab:appQuantTorus} we provide quantitative evaluation for the noisy and sparse input point cloud of the capped torus shape from \Cref{fig:bean}.

In \Cref{fig:OurZeros}, we show Marching Cubes results for the zero-level sets of our neural SDF on the set of shapes used in the quantitative evaluation in \Cref{fig:quantComp}. 

In \Cref{fig:additionalShapes}, we show the zero-level set results for some additional shapes. They were all computed with the same parameter set that is mentioned in the main text. Notably, no parameter tuning is required to reconstruct thin structures like the birds claws.

Finally, \Cref{tab:meshes} gives an overview over all considered models and the respective point cloud sizes. The point cloud for the capped torus experiment was created by ourselves.

\begin{figure}
	\centering
	\includegraphics[width=\linewidth]{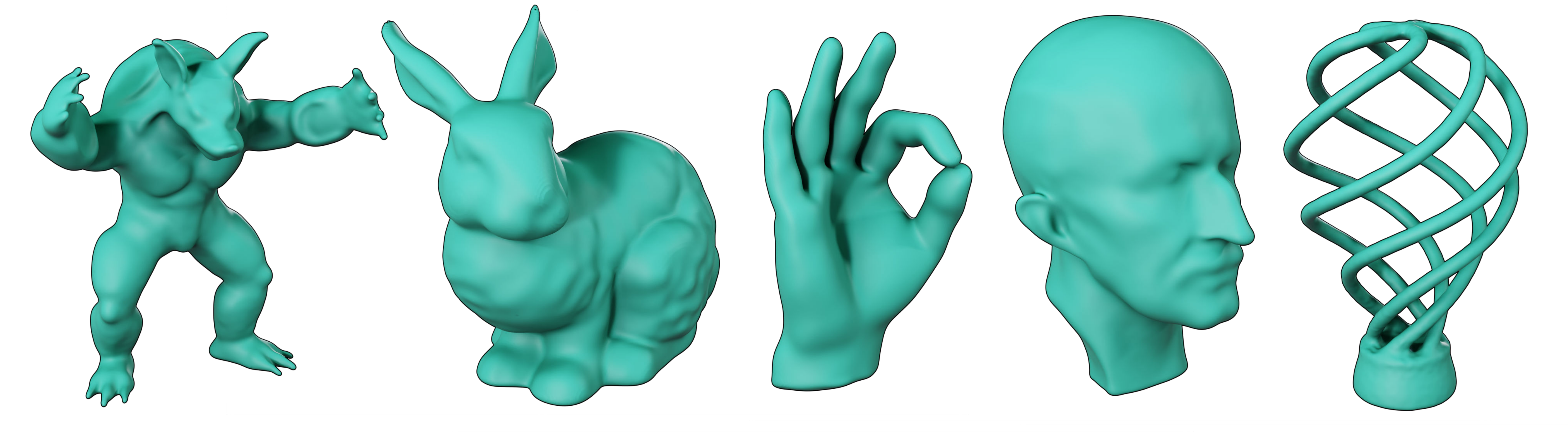}
	\caption{Marching Cubes result for zero-level sets of our neural SDF. Models used in the quantitative evaluation in \Cref{fig:quantComp}. \label{fig:OurZeros}}
\end{figure}
\begin{figure}
	\centering
	\includegraphics[width=\linewidth]{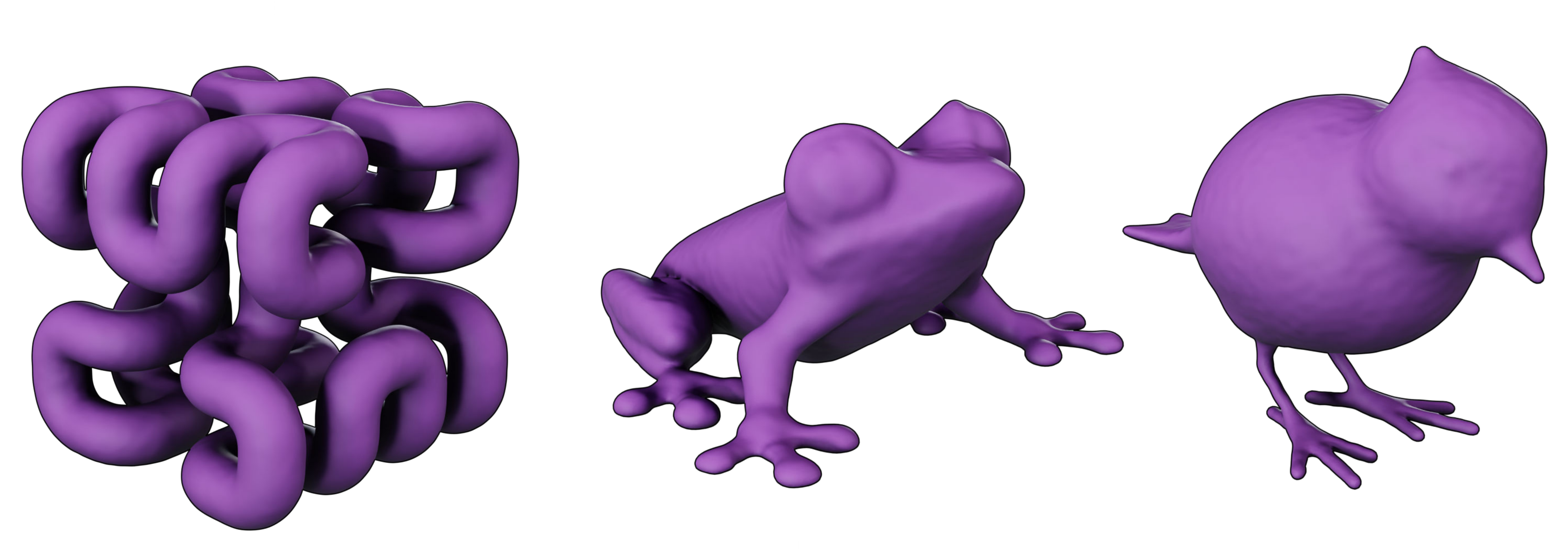}
	\caption{ Marching Cubes result for the zero-level sets of our neural SDF approximation.  \label{fig:additionalShapes}}
\end{figure}

\begin{center}
	\footnotesize
	\begin{tabular}{cccccc}
		\toprule
		&\textbf{Method} &$\mathbf{E_{recon}^{\mathcal{S}}}$ & $\mathbf{E_{eik}}$ &$\mathbf{E_{recon}^n}$ & $\mathbf{E_{SDF}}$\\
		\midrule
		\multirow{5}{*}{\rotatebox{90}{\textbf{Sparse }}}
		& SALD  &\(4.21\cdot 10^{-5}\) &   0.09983  & 0.00467&0.01214\\
		&Hessian& \(\mathbf{2.20\cdot 10^{-7}}\)& 0.73540 & \textbf{0.00004}&0.01866\\  
		&1Lip & \(1.06\cdot 10^{-4}\)&   \textbf{0.01046} &  0.00230&0.01043\\
		&HotSpot & \(5.10\cdot 10^{-7}\)&  0.04209 &  0.00012 &  0.00675\\
		&Ours & \(6.66\cdot 10^{-7}\) &   0.05366  &  0.00230 &\textbf{0.00240}\\ 
		\midrule
		\multirow{5}{*}{\rotatebox{90}{\textbf{Noise}}}
		& SALD & \(5.10\cdot 10^{-5}\)&  0.15230 &  0.02500&0.01054\\
		& Hessian& \(\mathbf{2.22\cdot 10^{-7}}\)&  0.73542 & \textbf{0.00004}&0.01867\\
		& 1Lip & \(2.13\cdot 10^{-5}\)&   \textbf{0.00775} & 0.00135&0.00426\\  
		& HotSpot& \(2.81\cdot 10^{-7}\)& 0.04043 &  0.00009 & 0.00596\\
		& Ours   & \(3.61\cdot 10^{-7}\)&  0.02514 & 0.00030 &\textbf{0.00210}\\   
		\bottomrule
	\end{tabular}
	\captionof{table}{Quantitative results for sparse and noisy input point clouds of the capped torus shape. \label{tab:appQuantTorus}}
\end{center}
\begin{table}[h]
	\centering
	\small
	\begin{tabular}{c c}
		\toprule
		\textbf{model} &point cloud size \\
		\midrule
		armadillo \cite{DBLP:conf/siggraph/KrishnamurthyL96} & 170k \\
		hand \cite{yeh2010template} & 6k\\
		Max Planck head \cite{Ivri2002divide} & 200k \\
		bunny \cite{DBLP:conf/siggraph/TurkL94} & 170k \\
		dog \cite{dyke2020shrec} & 100k \\
		lightbulb (id 39084) \cite{Thingi10k} & 500k\\
		bucky (id 41140) \cite{Thingi10k} & 100k\\
		pipes (id 53754) \cite{Thingi10k} & 100k\\
		bird (id 178340) \cite{Thingi10k}  & 100k \\
		frog (id 90736) \cite{Thingi10k} & 100k\\
		\bottomrule
	\end{tabular}
	\caption{Considered models and point cloud sizes. \label{tab:meshes}}
\end{table}

\end{document}